\documentclass[a4paper,12pt]{article}
\usepackage{amsmath,amsthm,amssymb,enumerate}
\usepackage{xcolor}
\usepackage{float}
\usepackage{multirow}
\usepackage{longtable}
\usepackage{diagbox}
\usepackage{soul}
\usepackage{cancel}
\usepackage{gensymb}
\usepackage{enumitem}
\usepackage[square,sort&compress,comma,numbers]{natbib}
\usepackage{bm}
\usepackage{hyperref}
\usepackage{subfig}
\usepackage{graphicx}
\usepackage[margin=1.7cm]{geometry}
\usepackage{tikz}
\usepackage{esint}
\usepackage{caption}
\usetikzlibrary{shapes,calc}
\usepackage{verbatim}
\usepackage{array}
\usepackage{bm}
\definecolor{maroon}{RGB}{144,0,32}
\newcolumntype{H}{>{\setbox0=\hbox\bgroup}c<{\egroup}@{}}

\usepackage{newtxtext,newtxmath}

\usepackage{multirow}
\usepackage{marginnote}
\chardef\bslash=`\\ 

\newcommand{\mathbi}[1]{{\boldsymbol #1}}
\def\N{\mathbb{N}}

\def\R{\mathbb{R}}

\def\err{\mathsf{err}}

\def\stab{\mathfrak{S}}

\def\dx{\,{\rm d}\x}

\def\<{\langle}
\def\>{\rangle}
\def\Poly{\mathbb{P}}

\def\div{{\rm div}}

\def\norm#1#2{\Vert#1\Vert_{#2}}

\newcounter{cst}

\def\bt{\begin{theorem}}
\def\et{\end{theorem}}
\def\bl{\begin{lemma}}
\def\el{\end{lemma}}
\def\bc{\begin{corollary}}
\def\ec{\end{corollary}}
\def\bd{\begin{definition}}
\def\ed{\end{definition}}
\def\br{\begin{remark}}
\def\er{\end{remark}}

\newcommand{\disc}{{\mathcal D}}

\def \hessian{\mathcal H}

\def \hd{\hessian_\disc}
\def \wdspace{H}

\def \symd{\R^{d\times d}}
\def \sym2{\R^{2\times 2}}
\def\cell{K}
\DeclareMathOperator*{\argmin}{argmin}

\newcommand{\mesh}{{\mathcal M}}

\newcommand{\edge}{{\sigma}}

\newcommand{\edges}{{\mathcal F}}

\newcommand{\edgesext}{{{\edges}_{\rm ext}}} 
\newcommand{\edgesint}{{{\edges}_{\rm int}}}

\newcommand{\x}{\mathbi{x}}

\newcommand{\vertices}{{\mathcal V}}

\newcommand{\be}{\begin{equation}}
\newcommand{\ee}{\end{equation}}

\renewcommand{\O}{\Omega}

\renewcommand{\d}{{\,\rm d}}

\newcommand{\ba}{\begin{array}{llll}   }
\newcommand{\bac}{\begin{array}{c}}
\newcommand{\bari}{\begin{array}{r}}
\newcommand{\ea}{\end{array}}

\newcommand{\NORM}[1]{{\left\vert\kern-0.25ex\left\vert\kern-0.25ex\left\vert #1 
    \right\vert\kern-0.25ex\right\vert\kern-0.25ex\right\vert}}

\newcommand{\tPsi}{\widetilde{\Psi}}

\newcommand{\dm}{\disc_m}  
\def\hat#1{\widehat{#1}}
\newcommand{\bsymb}[1]{{\boldsymbol #1}}
\newcommand{\bfX}{\bsymb{X}}
\newcommand{\bH}{\bsymb{H}}

\newcommand{\bfXd}{\bsymb{X}_{\disc,0}}
\newcommand{\bL}{\bsymb{L}}
\newcommand{\cL}{\mathcal{L}}
\newcommand{\cA}{\mathcal{A}}
\newcommand{\cB}{\mathcal{B}}
\newcommand{\fl}{\;\mbox{ for all } \;}
\usepackage{graphicx}
\usepackage{epstopdf}
\usepackage{pdfpages}
\usepackage{subfig,epsfig}
\usepackage{textcomp}
\usepackage{color}
\usepackage{tikz}
\usepackage{caption}

\usepackage{cancel}

\usetikzlibrary{arrows}
\usetikzlibrary{shapes,calc}

\usepackage[square,sort&compress,comma,numbers]{natbib}

\def\assum#1{{\rm\textbf{(A#1)}}}
\def\propty#1{{\rm\textbf{(P#1)}}}

\newtheorem{theorem}{Theorem}[section]
\newtheorem{remark}[theorem]{Remark}

\newtheorem{lemma}[theorem]{Lemma} 
\newtheorem{definition}[theorem]{Definition}

\newtheorem{corollary}[theorem]{Corollary}

\numberwithin{equation}{section}

\def\XXint#1#2#3{{\setbox0=\hbox{$#1{#2#3}{\int}$ }
\vcenter{\hbox{$#2#3$ }}\kern-.6\wd0}}

\usepackage[normalem]{ulem}
\newcounter{corr}
\definecolor{violet}{rgb}{0.580,0.,0.827}
\newcommand{\corr}[3]{\typeout{Warning : a correction remains in page
\thepage}
				\stepcounter{corr}        
				{\color{blue}\ifmmode\text{\,\sout{\ensuremath{#1}}\,}\else\sout{#1}\fi}
       {\color{red}#2}
       {\color{violet} #3}}

\newcounter{cexp}
\catcode`\@=11
\def\terml#1{T_{\refstepcounter{cexp}\@bsphack
\protected@write\@auxout{}%
           {\string\newlabel{#1}{{\thecexp}{\thepage}}}\thecexp}}
\catcode`\@=12

\begin{document}
\title{Convergence Analysis of the Hessian Discretisation Method for Fourth Order Semi-linear Elliptic Equations with General Source}
\author{Devika Shylaja \footnote{Department of Mathematics, Birla Institute of Technology and Science, Pilani, K K Birla Goa Campus, Zuarinagar, Sancoale, Goa 403726, India. {\tt devikas@goa.bits-pilani.ac.in}}}
\maketitle
\begin{abstract}
This paper presents a convergence analysis for the Hessian Discretisation Method (HDM) applied to fourth-order semilinear elliptic equations involving a trilinear nonlinearity and general source, based on two complementary approaches. The HDM serves as a unified framework for the convergence analysis of various numerical schemes, including conforming and nonconforming finite element methods (ncFEMs) and gradient recovery (GR) based methods. Error estimates for the Adini ncFEM and GR methods are derived for the first time, which provide an explicit order of convergence. The analysis relies on four key HDM properties along with a suitable companion operator to establish convergence results. Moreover, a convergence analysis is developed within the HDM framework, which does not require additional regularity assumptions on the exact solution or the assumption that the exact solution is regular. The paper further discusses two significant applications: the Navier--Stokes equations in stream function--vorticity formulation and the von K\'{a}rm\'{a}n equations for plate bending. Numerical experiments are provided to demonstrate the performance of the GR method, Morley, and Adini ncFEMs. 
\end{abstract}

\medskip

\noindent{\bf Keywords:} {Hessian discretisation, Navier Stokes equations, von K\'{a}rm\'{a}n equations, plate bending, non-linear equations, finite element,  Adini, gradient recovery, error estimates, general source}

\section{Introduction}
Fourth-order elliptic partial differential equations, both linear and nonlinear, serve as fundamental components in various physical models, most notably in the structural analysis of plates, beams, and shells \cite{ciarlet1978finite,ciarlet_plate,NS_FE79}. The present work investigates convergence analysis for numerical approximations of fourth-order semilinear problems characterized by trilinear nonlinearity and a general source term with clamped boundary conditions. This study utilizes the Hessian Discretisation Method (HDM) \cite{HDM_linear, DS_HDM, HDM_nonlinear, DS_HDMcontrol}, a unified framework designed for such fourth-order problems.

\medskip

The HDM approach is advantageous because it offers a centralized platform for evaluating a various numerical methods within a single framework. This includes conforming finite element methods (FEMs), Morley and Adini non conforming finite element methods (ncFEMs), and methods based on gradient recovery (GR) operators \cite{ciarlet1978finite, BL_FEM, BL_MFEM}. This framework operates on a Hessian discretisation quadruple, which replaces continuous functions and their derivatives (gradient and Hessian) with discrete approximations. By verifying four core properties - coercivity, consistency, limit-conformity, and compactness—one can systematically prove the convergence of different schemes across various linear \cite{HDM_linear} as well as semilinear \cite{HDM_nonlinear} models for data in $L^2(\Omega)$. This paper extends the work presented in \cite{HDM_nonlinear}.


\medskip

%

The specific abstract problem addressed here has significant practical implications, particularly for the stream function–vorticity formulation in 2D incompressible Navier–Stokes equations \cite{BRR,NS_FE79} and the von K\'arm\'an equations for elastic plate bending \cite{ciarlet_plate}. The stream-function formulation of the incompressible Navier–Stokes equations eliminates the pressure variable and automatically enforces mass conservation. However, this reformulation converts the original second-order system into a fourth-order nonlinear equation for the scalar stream function. Various numerical approaches have been developed to approximate this formulation; for instance,  $C^0$ FEM \cite{Cayco1986}, $C^1$ conforming FEMs \cite{Foster2013}, a $C^0$ weak Galerkin FEM \cite{Zhang2020_C0WG_NS}, and virtual element methods \cite{Adak2021,Zhang2024_NS}. Previous numerical studies on von Kármán equations have utilized several methods, including conforming FEMs \cite{Brezzi,ng1}, Morley ncFEM \cite{ng2}, mixed FEMs \cite{BRR,TM76}, $C^0$ interior penalty methods \cite{SBMNARLS}, discontinuous Galerkin methods \cite{CCGMNN18}, and virtual element methods \cite{Shylaja2024MorleyVEM, Lovadina2021VEMVK}. Recently, a unified a priori analysis of four second-order finite element methods, including Morley ncFEM, for rough data has been developed for a fourth-order quadratic semilinear problem \cite{Carstensen2023Unified}, which, in particular, applies to the two applications mentioned above. Note that the weak form of the von K\'arm\'an equations considered there differs from the formulation studied in this paper. Furthermore, in \cite{Carstensen2023Unified}, convergence is established by assuming additional regularity of the exact solution and the well-posedness of the linearized problem around the exact solution. 

\medskip

In this paper, we employ two complementary approaches for the convergence analysis.  The first approach extends the compactness-based framework established in \cite{HDM_nonlinear}. While prior research in \cite{HDM_nonlinear} established convergence of HDM via compactness techniques for $L^2(\Omega)$ data without assuming extra regularity of the exact solution or the assumption that the exact solution is regular, this paper expands that scope for rough data. In this context, the solution to the continuous weak formulation is identified as the limit of a sequence of discrete solutions, which also proves the existence of a continuous solution.  The second approach provides explicit orders of convergence through error estimates on convex domains. This analysis assumes that the exact solution possesses $H^3$ regularity and that the linearized problem around the exact solution is well-posed. These two approaches complement each other. To establish these two approaches, we utilize a smoothing (companion) operator alongside the four fundamental properties of the HDM. To the best of our knowledge, the convergence of the HDM framework using the first approach has not previously been applied to rough data. Furthermore, error analysis for Adini ncFEM and GR-based methods for nonlinear fourth-order problems, that provides explicit orders of convergence, represents a significant and novel contribution of this work.

\medskip

The contributions of this article are the following:
\begin{itemize}
	\item  { A unified framework} provided by HDM for fourth order semi-linear elliptic equations with a trilinear nonlinearity and a general source term that applies to several numerical methods such as conforming FEMs, Adini and Morley ncFEMs, and GR methods.

\item Wellposedness of discrete solution for general data in $H^{-2}(\Omega)$ in an abstract framework, assuming the existence of a $C^1$-conforming companion operator and a specific property of HDM.

		\item {Convergence analysis} by compactness techniques for rough data in $H^{-1}(\Omega)$ in HDM framework without any extra-regularity assumption on the exact solution that employs only four properties along with a suitable companion operator.

		\item {Error estimates} in an abstract setting for rough data in $H^{-1}(\Omega)$ on convex domains that establish orders of convergence in $L^2$, $H^1$, and $H^2$ like norms through three properties of HDM and a companion operator.

\item Convergence of the Newton's method for the discrete problem.

		\item {Applications} to the stream function vorticity formulation of 2D Navier--Stokes equation and the von K\'{a}rm\'{a}n equations.
	\item Numerical results of Navier--Stokes equation and von K\'{a}rm\'{a}n equations using GR method, Morley ncFEM and Adini ncFEM to confirm the theoretical estimates.
\end{itemize}

\medskip

The paper is organised as follows. Section~\ref{sec.modelproblemHDM} outlines the abstract weak formulation for fourth-order semilinear elliptic equations and the HDM framework for rough data, highlighting key results for GR method and Adini ncFEM alongside the four HDM properties and the companion operator. Section~\ref{sec.convergence} establishes the convergence of the HDM based on compactness techniques for rough data, while Section~\ref{sec:error} derives explicit error estimates to establish the order of convergence.  Section~\ref{sec:HDMegandapplications} explores various numerical methods that align with this framework, applying it to the 2D incompressible Navier-Stokes and von K\'{a}rm\'{a}n equations. Section~\ref{sec:numericalresults} presents numerical results that support the theoretical analysis.  Finally, some concluding remarks are given in Section~\ref{sec:conclusion}.

\medskip

\textbf{Notations}. Let $\O \subset \R^d$ $(d\ge1)$ be a bounded domain with boundary $\partial \Omega$ and let the outer normal be denoted by $n$. For brevity, we follow the Einstein summation convention that implies summation over a set of indexed terms in a product of vectors, tensors or differential operators unless otherwise stated. The scalar product on $\symd$ is defined by $\xi:\phi=\xi_{ij}\phi_{ij}$. For a function $\xi : \O \rightarrow \symd$, denoting the Hessian operator by $\hessian$, set $\hessian : \xi = \partial _{ij}\xi_{ij}$. For $a,b\in\R^d$, let $a \otimes b$ denotes the 2-tensor with coefficients $a_ib_j$. The standard $L^2$ inner product and norm (applied on $L^2(\O)$, $L^2(\O;\R^d)$, and $L^2(\O;\symd)$) are denoted by $(\cdot,\cdot)$ and $\|{\cdot}\|$. The semi-norm and norm in $H^m(\Omega)$ (resp. $W^{m,p} (\Omega)$,  $1 \le p \le \infty$), are denoted by $|\cdot|_{m}$ and $\|\cdot\|_{m}$ (resp.  $|\cdot|_{m,p}$ and $\|\cdot\|_{m,p}$). Let $k\ge 1$ be an integer and, for $E$ a vector space, set $\bsymb{E}=E^k$. For simplicity of notation, we use the same notation to denote the norms in $\bsymb{E}$ and $E$. The duality pairing between $\bsymb{H}^{-k}(\O)$ and $\bsymb{H}^{k}_0(\O)$ is denoted by $\langle \cdot,\cdot \rangle_{-k,k}$. The norm on $\bsymb{H}^{-k}(\Omega)$ is denoted by $\|\cdot\|_{-k}$.

\section{Model Problem and Hessian Discretisation}\label{sec.modelproblemHDM}

This section presents the abstract setting of the weak formulation of semi-linear fourth-order elliptic problems with a trilinear nonlinearity and a general source term. This is followed by a discretization based on the Hessian discretisation method. The main results of this article are stated for the Adini ncFEM and GR method for ease of exposition. The properties that ensure the convergence analysis are discussed at the end of this section.

\medskip

\subsection{Abstract Problem}

The continuous abstract problem seeks $\Psi \in \bfX$ with $X:=H^2_0(\O)$ such that
\be \label{abstract_weak}
\mathcal{A}(\hessian \Psi,\hessian \Phi)+\mathcal{B}(\hessian \Psi, \nabla \Psi, \nabla \Phi)=\mathcal{L}(\Phi)  \fl \Phi \in \bfX,
\ee
where $\hessian \Psi$ and $\nabla \Psi$ are to be understood component-wise, that is: for $\Psi=(\psi_1,\cdots,\psi_k)$, $\hessian \Psi=(\hessian \psi_1,\cdots,\hessian \psi_k)$ and $\nabla \Psi=(\nabla \psi_1,\cdots,\nabla \psi_k)$. Assume that 
\begin{itemize}
	\item[\assum{1}] $\mathcal{A}(\cdot,\cdot)$ is a  continuous and coercive bilinear form on $\bsymb{L}^2(\O;\symd) \times \bsymb{L}^2(\O;\symd)$.
	\item[\assum{2}] $\mathcal{B}(\cdot,\cdot,\cdot)$ is a continuous trilinear form on $\bsymb{L}^2(\O;\symd) \times \bsymb{L}^4(\O;\R^d) \times \bsymb{L}^4(\O;\R^d)$.
	\item[\assum{3}] $\mathcal{B}(\Xi,\Theta,\Theta)=0$ for all $ \Xi \in \bsymb{L}^2(\O;\symd)$ and $\Theta \in \bsymb{L}^4(\O;\R^d)$.
	\item[\assum{4}] $\mathcal{L}(\cdot)$ is a continuous linear functional on $\bsymb{H}^{2}_0(\Omega)$, that is, $\mathcal{L} \in \bsymb{H}^{-2}(\Omega)$. 
\end{itemize}
The abstract formulation \eqref{abstract_weak} covers the stream function vorticity formulation of the incompressible 2D Navier--Stokes problem \cite{Lions_NS,BRR} for $k=1$ as well as the von K\'{a}rm\'{a}n equations \cite{ciarlet_plate} for $k=2$, see Section~\ref{sec:applications} for more details.

\begin{remark}[Comparison with \cite{HDM_nonlinear}]
This work extends the results in \cite{HDM_nonlinear} in two significant ways. First, we consider more general data, whereas \cite{HDM_nonlinear} was restricted to data in $\bsymb{L}^2(\Omega)$. Second, beyond providing a compactness-based convergence analysis for rough data, we establish error estimates that determine the order of convergence within the HDM framework - an aspect not addressed in \cite{HDM_nonlinear}.
\end{remark}

\begin{remark}[Assumption $\assum{4}$]
While assumptions $\assum{1}$–$\assum{4}$ ensure the well-posedness of the discrete problem within the HDM framework (Theorem~\ref{thm:wellposedness}), the compactness-based convergence analysis requires $\assum{1}$–$\assum{3}$, with $\assum{4}$ limited to $\bsymb{H}^{-1}(\Omega)$ (Theorem~\ref{convergence_abstract_non-linear}). This same restricted version of $\assum{4}$ is also a prerequisite for establishing the order of convergence in the error estimates (Theorem~\ref{thm.err}).
\end{remark}

\subsection{Hessian Discretisation Method} \label{sec.HDM}
This section presents the HDM for fourth-order nonlinear elliptic equations with rough data. The HDM framework originally developed in \cite{HDM_nonlinear} for $\bsymb{L}^2$ data is thus extended to accommodate this more general case. 
	
	\begin{definition}[Hessian discretisation]\cite[Definition 3.1]{HDM_nonlinear}\label{HD}~
	A Hessian discretisation (HD) for fourth order non-linear elliptic equations with clamped boundary conditions is a quadruplet $\disc=(X_{\disc,0},\Pi_\disc,\nabla_\disc,\hd)$ such that
	\begin{itemize}
		\item  $X_{\disc,0}$ is a finite dimensional real vector space, 
		\item the linear mapping $\Pi_\disc:X_{\disc,0}  \rightarrow L^2(\O)$ gives a reconstructed discrete function in $L^2(\O)$ from vectors in $X_{\disc,0},$
		\item the linear mapping $\nabla_\disc:X_{\disc,0}  \rightarrow L^4(\O;\R^d)$ gives a reconstructed discrete gradient in $L^4(\O;\R^d)$ from vectors in  $X_{\disc,0}$,
		\item the linear mapping $\hd:X_{\disc,0}  \rightarrow L^2(\O;\R^{d \times d})$ gives a reconstructed discrete version of Hessian in $L^2(\O;\R^{d \times d})$ from $X_{\disc,0}$. The operator $\hd$ is such that $\norm{{\cdot}}{\disc}=:\norm{\hd \cdot}{}$ is a norm on $X_{\disc,0}.$
	\end{itemize}
\end{definition}
To accommodate general data, the existence of a $C^1$-conforming {\it companion/smoothing} operator is assumed. This operator maps elements from the discrete space $X_{\disc,0}$ to $X$ and satisfies the conservation properties stated below.

\smallskip

\assum{5} (Companion operator) {\it There exists a linear map $E_\disc :X_{\disc,0} \rightarrow X$ called the companion operator. Let
\begin{subequations}\label{sup.ED.est}
\begin{align}
	&\delta(E_\disc):=\sup_{\psi_\disc \in X_{\disc,0}\setminus{\{0\}} }\frac{\norm{\Pi_\disc \psi_\disc-E_\disc \psi_\disc}{}}{\norm{\psi_\disc}{\disc}}\,,\label{sup.ED.est.delta}\\
	&\omega(E_\disc):=\sup_{\psi_\disc  \in X_{\disc,0}\setminus{\{0\}}}\frac{\norm{\nabla_\disc \psi_\disc-\nabla E_\disc \psi_\disc}{}}{\norm{ \psi_\disc}{\disc}}\,,\label{sup.ED.est.omega1}\\
&\Gamma(E_\disc):=\sup_{\psi_\disc \in X_{\disc,0}\setminus{\{0\}}}\frac{\norm{\hessian E_\disc \psi_\disc}{}}{\norm{ \psi_\disc}{\disc}} \label{sup.ED.est.Gamma}.
\end{align}
\end{subequations}
}
The companion operators associated with the sequence of HDs $(\disc_m)_{m\in\mathbb{N}}$ are expected to yield $\delta(E_{\disc_m}) \to 0$ and $\omega(E_{\disc_m}) \to 0$ as $m \to \infty$, with $\Gamma(E_{\disc_m})$ remaining bounded for all $m$. A more detailed discussion of this operator appears in Section~\ref{sec:companion}.

\medskip

Let $\disc=(X_{\disc,0},\Pi_\disc,\nabla_\disc,\hd)$ be a HD in the sense of Definition \ref{HD} and let $E_\disc$ be a companion operator. The associated numerical scheme, referred to as the Hessian scheme (HS), for \eqref{abstract_weak} seeks $\Psi_\disc \in \bfXd:=X_{\disc,0}^k$ such that
\begin{align}
&\mathcal{A}(\hd \Psi_\disc,\hd \Phi_\disc)+\mathcal{B}(\hd \Psi_\disc, \nabla_\disc \Psi_\disc, \nabla_\disc \Phi_\disc)=\mathcal{L}(E_\disc\Phi_\disc) \fl  \Phi_\disc \in \bfXd,\label{abstract_HS}
\end{align}
where $\hd \Phi_\disc$, $\nabla_\disc \Phi_\disc$ and $E_\disc \Phi_\disc$ act component-wise in the sense that if $\Phi_\disc=(\phi_{\disc,1}, \cdots,\phi_{\disc,k})$ and $F_\disc\in\{\Pi_\disc,\nabla_\disc,\hd, E_\disc\}$, then $F_\disc \Phi_\disc=(F_\disc\phi_{\disc,1},\cdots,F_\disc\phi_{\disc,k})$.
\medskip

The right-hand side of \eqref{abstract_HS} incorporates $E_\disc$, which ensures that \eqref{abstract_HS} is well-defined even when the approximation space is not a subspace of $\bfX$. This is particularly relevant for methods like the Morley ncFEM.

\begin{remark}[Simplification for $\bsymb{L}^2$ data]\label{rem.L2}
When the data is in $\bsymb{L}^2(\Omega)$, the right-hand side of \eqref{abstract_HS} is replaced by $\mathcal{L}(\Pi_\disc \Phi_\disc)$. In this specific case, the use of $E_\disc$ is not required. Consequently, the convergence analysis via the compactness argument remains consistent with the results in \cite{HDM_nonlinear}.
\end{remark}


\subsection*{Main results} 

The following result states the 
\begin{itemize}
\item[$(i)$] wellposedness of the discrete problem for data in $\bH^{-2}(\Omega)$
\item[$(ii)$] convergence by compactness arguments that does not require any smoothness of the solution for data in $\bH^{-1}(\Omega)$
\item[$(iii)$] convergence under the assumption that the linearized problem around the exact solution is well-posed (regular solution) that provides order of convergence for data in $\bH^{-1}(\Omega)$.  
\end{itemize}
for the Adini ncFEM and GR method. The results are proved in Theorems \ref{thm:wellposedness}, \ref{convergence_abstract_non-linear}, and \ref{thm.err} in the unified HDM framework that in particular covers conforming, Morley and Adini ncFEMs, and GR methods. Furthermore, this framework applies to two specific cases: stream function vorticity formulation of the incompressible 2D Navier--Stokes problem and von K\'{a}rm\'{a}n equations. Let the assumptions $\assum{1}-\assum{3}$ hold.

\medskip

\textbf{(i) Wellposedness of the discrete problem:} Under assumption $\assum{4}$, there exists at least one solution $\Psi_\disc \in \bfXd$ to \eqref{abstract_HS}. Furthermore, the solution is unique provided the data is sufficiently small.
\medskip

\textbf{(ii) Convergence by compactness:} Let assumption $\assum{4}$ be restricted to $\bsymb{H}^{-1}(\Omega)$. For each $m \in \mathbb{N}$, there exists at least one solution $\Psi_{\disc_m}\in \bfX_{\disc_m,0}$ to \eqref{abstract_HS} on $\disc_m$. As $m \rightarrow \infty$, there exists a subsequence of $(\disc_m)_{m \in \mathbb{N}}$ (denoted using the same notation) and a solution $\Psi \in \bfX$ to the abstract problem \eqref{abstract_weak} such that  the following convergence hold:
\begin{itemize}
\item $\Pi_{\disc_m}\Psi_{\disc_m} \rightarrow \Psi$ in $\bsymb{L}^2(\Omega)$,
\item $\nabla_{\disc_m} \Psi_{\disc_m} \rightarrow \nabla \Psi$ in $\bsymb{L}^4(\Omega;\mathbb{R}^d)$,
\item $\hessian_{\disc_m} \Psi_{\disc_m} \rightarrow \hessian \Psi$ in $\bsymb{L}^2(\Omega; \mathbb{R}^{d \times d})$. 
\end{itemize}
Moreover, the uniqueness of the discrete solution holds for sufficiently small data.

\medskip

The HD $\disc_m$ is usually associated with a mesh $\mesh_{h_m}$ whose size is denoted by $h_m$.  As $m \to \infty$, the mesh is refined and thus $h_m \to 0$. 

\medskip

$\textbf{(iii) Convergence via error estimates:}$ Let $\Omega$ be a convex domain and let $\mathcal{L}$ in $\assum{4}$ be restricted to $\bsymb{H}^{-1}(\Omega)$. Suppose $\Psi \in \bsymb{H}^3(\Omega) \cap \bfX$ is a regular solution to \eqref{abstract_weak}. Let $P_\disc: \bfX \to  \bfX_{\disc,0}$ be an interpolant operator. For a sufficiently small discretization parameter $h$, there exists a discrete solution $\Psi_\disc$ to \eqref{abstract_HS} that satisfies  $\norm{ \Psi_\disc-P_\disc \Psi}{\disc}\lesssim h$ as well as the following error estimates:
$$\|\Pi_\disc \Psi_\disc - \Psi\|\lesssim h,\,\| \nabla_\disc \Psi_\disc -\nabla \Psi\| \lesssim h,\quad \mbox{ and }\quad \| \hessian_{\disc}\Psi_\disc - \hessian\Psi\|\lesssim h.
$$
Moreover, local uniqueness of the solution is guaranteed for a sufficiently small choice of $h$.
\subsubsection{Properties of HDM}\label{sec.properties}
This section deals with the four properties associated with an HD in the sense of Definition \ref{HD}, which are critical for the convergence analysis of a HS.

\medskip

The first quantity is a constant, $C_\disc$, that ensures discrete Poincar\'e inequalities. It is defined by 
\be\label{def.CD}
C_\disc := \max_{w_\disc\in X_{\disc,0}\setminus\{0\}} \left(\max\left\{\frac{\norm{\Pi_\disc w_\disc}{}}{\norm{w_\disc}{\disc}},
\frac{\norm{\nabla_\disc w_\disc}{0,4}}{\norm{w_\disc}{\disc}}\right\}\right).
\ee

The second quantity is the interpolation error $S_\disc$ defined by: for all $\varphi\in H^2_0(\O)$,
\be\label{def.SD}
\begin{aligned}
	&S_\disc(\varphi):=\min_{w_\disc\in X_{\disc,0}}\Big(\norm{\Pi_\disc w_\disc-\varphi}{}
	+\norm{\nabla_\disc w_\disc-\nabla\varphi}{0,4}	+\norm{\hd w_\disc-\hessian \varphi}{}\Big).
\end{aligned}
\ee
To define the limit-conformity measure for the HS, introduce  $$\wdspace(\O):=\left\{\xi\in L^2(\O;\R^{d \times d})\,;\,\hessian:\xi \in L^2(\O) \right\}$$ and $$ H_{\rm{div}}({\O}):=\left\{\phi \in L^2(\O;\R^d):\, \mbox{div}\phi \in L^2(\O)\right\}.$$
 For all $\xi \in \wdspace(\O)$ and $\phi \in H_{\rm{div}}({\O})$, set
\be\label{def.WD}
\begin{aligned}
	&W_\disc(\xi):=\max_{w_\disc\in X_{\disc,0}\backslash\{0\}}
	\frac{1}{\norm{w_\disc}{\disc}}\Big|\int_\O \Big((\hessian:\xi)\Pi_\disc w_\disc - \xi:\hd w_\disc \Big)\d\x \Big|, 
\end{aligned}
\ee
\be\label{def.WD_gradient}
\begin{aligned}
	& \hat{W}_\disc(\phi):=\max_{w_\disc\in X_{\disc,0}\backslash\{0\}}
	\frac{1}{\norm{w_\disc}{\disc}}\Big|\int_\O \Big(\nabla_\disc w_\disc \cdot \phi + \Pi_\disc w_\disc \mbox{ div} \phi \Big)\d\x \Big|. \end{aligned}
\ee
The limit-conformity measures, $\hat{W}_{\disc}$ and $W_{\disc}$, quantify the defects of the discrete (single and double) integration-by-parts formulas. These measures evaluate the defect of conformity between the reconstructed function and its corresponding reconstructed gradient and Hessian.

\begin{definition}[Coercivity, consistency, limit-conformity and compactness]\label{def:coercive} \cite[Definition 4.1]{HDM_nonlinear}
	Let $(\disc_m)_{m \in \N}$ be a sequence of HDs in the sense of Definition \ref{HD}. We say that
	\begin{itemize}
		\item[$(i)$] $(\disc_m)_{m\in\N}$ is \emph{coercive} if there exists $C_P \in \R^+$ such that $C_{\disc_{m}} \leq C_P$ for all $m \in \N$.
		\item[$(ii)$] $(\disc_m)_{m\in\N}$ is \emph{consistent}, if for all $ \varphi\in H^2_0(\O)$,
		\be \nonumber 
		\lim_{m \rightarrow \infty} S_{\disc_m}(\varphi)=0.
		\ee
		\item[$(iii)$] $(\disc_m)_{m\in\N}$ is \emph{limit-conforming}, if for all $ \xi \in H(\O)$ and for all $\phi \in  H_{\rm{div}}({\O})$,
		\be \nonumber
			\lim_{m \rightarrow \infty} \big(W_{\disc_m}(\xi)+\hat{W}_{\disc_m}(\phi)\big)=0.
		\ee
		\item[$(iv)$] $(\disc_m)_{m \in \N}$ is \emph{compact} if for any sequence $(u_m)_{m \in \N}$ such that $u_m \in X_{\disc_m,0}$ and $(\norm{ u_m}{{\disc_m}})_{m \in \N}$ is bounded, the sequence $(\Pi_{\disc_m}u_m)_{m \in \N}$ is relatively compact in $L^2(\O)$, and the sequence $(\nabla_{\disc_m}u_m)_{m \in \N}$ is relatively compact in $L^4(\O;\R^d)$.
	\end{itemize}
\end{definition}
Numerical methods that satisfy the four properties outlined in Definition \ref{def:coercive} include conforming FEMs, Morley and Adini ncFEMs, and GR methods. The specific properties for each of these approaches are detailed in Section \ref{sec:HDMeg}.

\subsubsection{Companion Operator}\label{sec:companion}

In addition to the standard properties discussed in Section~\ref{sec.properties}, the companion operator $E_\disc$ (defined in $\assum{5}$) is necessary to ensure the convergence of the HDM. It will be expected that, along the considered sequence $(\disc_m)_{m\in\N}$ of HDs, the corresponding companion operators will be such that $\delta(E_{\disc_m})\to 0$, $\omega(E_{\disc_m}) \to 0$ as $m \to \infty$, and $\Gamma(E_{\disc_m})$ remains bounded for all $m$. Specifically, the corresponding companion operators must satisfy the following conditions based on the level of analysis.
\begin{itemize}
\item  {\it Well-posedness for $\bH^{-2}(\Omega)$ data: }To ensure the discrete problem is well-posed for data in $\bH^{-2}(\Omega)$, it is sufficient for the companion operator $E_\disc$ to have a bounded $\Gamma(E_\disc)$.

\item {\it Convergence via compactness for $\bH^{-1}(\Omega)$ data: }The establishment of convergence through compactness techniques for $\bH^{-1}(\Omega)$ data requires that, for the sequence $(\disc_m)_{m\in\mathbb{N}}$ of HDs, the corresponding companion operators satisfy $\omega(E_{\disc_m})\to 0$ as $m \to \infty$.

\item {\it Error estimates for $\bH^{-1}(\Omega)$ data: }In contrast, the derivation of error estimates necessitates a broader set of conditions: $\delta(E_{\disc_m})\to 0$ and $\omega(E_{\disc_m}) \to 0$, alongside the boundedness of $\Gamma(E_{\disc_m})$.
\end{itemize}



 For conforming FEMs, $E_\disc$ is nothing but the identity operator and hence, 
\begin{equation}
\delta (E_\disc)=\omega(E_\disc)=0, \mbox{ and }\Gamma(E_\disc)=1.\nonumber
\end{equation}
 An explicit companion operator for the Morley ncFEM \cite{Morley_plate}  and Adini ncFEM \cite{Brenner_ncfem} satisfies
\begin{equation}
\delta(E_\disc)=h^2,\,\omega(E_\disc)=h, \mbox{ and }\Gamma(E_\disc)\le C,\nonumber
\end{equation}
 where $C$ is a positive constant independent of the mesh-size $h$. Recently, a companion operator for the GR method has been constructed in \cite{DS_HDMcontrol} with 
\begin{equation}
\delta(E_\disc)=h,\,\omega(E_\disc)=0, \mbox{ and }\Gamma(E_\disc)\le C.\nonumber
\end{equation}


\section{Convergence by compactness}\label{sec.convergence}

This section establishes the convergence of the HS, provided the underlying sequences of HDs satisfy the properties in Definition \ref{def:coercive} and the companion operator $E_\disc$ in $\assum{5}$. This convergence is proved without any extra regularity assumption on the exact solution, or the assumption that the linearized problem around this solution is well-posed.
\medskip

The following two lemmas provide the auxiliary results required to establish our main convergence theorem.

\begin{lemma}[Regularity of the limit] \label{regularity_nonlinear}
	\cite[Lemma 5.1]{HDM_nonlinear}Let $(\disc_m)_{m \in \N}$ be a  coercive and limit-conforming sequen\-ce of HDs in the sense of Definition \ref{def:coercive}$(i)$ and $(iii)$. Let $u_m\in X_{\disc_m,0}$ be such that $\norm{u_m}{\disc_m}$ remains bounded. Then, there exists a subsequence of $(\disc_m,u_m)_{m \in \N}$ (denoted using the same notation)  and $u \in H^2_0(\O)$ such that $\Pi_{\disc_m}u_m$ converges weakly to $u$ in $L^2(\O)$, $\nabla_{\disc_m}u_m$ converges weakly to $\nabla u$ in $L^4(\O;\R^d)$, and $\hessian_{\disc_m}u_m$ converges weakly to $\hessian u$ in $L^2(\O; \R^{d \times d})$.
\end{lemma}

\begin{lemma}\label{abstract_trilinear_cv}
Let $\Xi_m \rightarrow \Xi$ weakly in $\bsymb{L}^2(\O;\R^{d\times d}), \,\Theta_m \rightarrow \Theta$ in $\bsymb{L}^4(\O;\R^d)$  and $X_m \rightarrow X$ in $\bsymb{L}^4(\O;\R^d)$ as $m \rightarrow \infty$. Assume that $\mathcal B:\bsymb{L}^2(\O;\R^{d\times d})\times \bsymb{L}^4(\O;\R^d)\times \bsymb{L}^4(\O;\R^d)\to\R$ is a continuous trilinear form. Then,
	$\mathcal{B}(\Xi_m, \Theta_m, X_m) \rightarrow \mathcal{B}(\Xi, \Theta, X)$ as $m \rightarrow \infty.$
\end{lemma}	
\begin{proof}
Simple manipulation leads to
\begin{align*}
\mathcal{B}(\Xi_m, \Theta_m, X_m)- \mathcal{B}(\Xi, \Theta, X)&=\mathcal{B}(\Xi_m-\Xi, \Theta_m, X_m) +\mathcal{B}(\Xi, \Theta_m-\Theta, X_m)\\
&\qquad+\mathcal{B}(\Xi, \Theta, X_m-X).
\end{align*}
The continuity of $\mathcal{B}(\cdot,\cdot,\cdot)$, combined with the weak/strong convergence, yields the conclusion.
\end{proof}
Let $C_{1,\rm eq}$ denotes the constant such that 
\begin{equation}\label{defn.C1}
\|\phi\|_{1} \le C_{1,\rm eq}\, | \phi|_1 
\quad \text{for all } \phi \in H^1_0(\Omega).
\end{equation}
Let $C_{2,\rm eq}$ denote the constant such that
\begin{equation}\label{defn.C2}
\|\varphi\|_{2} \le C_{2,\rm eq}\,| \varphi|_2
\quad \text{for all } \varphi \in H^2_0(\Omega).
\end{equation}
Let  $\overline{\alpha}$ denote the coercivity constant of $\mathcal{A}(\cdot,\cdot)$ and let $\norm{\mathcal B}{}$ denote the norm of $\mathcal B$ on $\bsymb{L}^2(\O;\symd) \times \bsymb{L}^4(\O;\R^d) \times \bsymb{L}^4(\O;\R^d)$. From now onwards, $a\lesssim b$ { means that } $a\le Cb$ for some generic constant $C$ that depends only on the data of the continuous model \eqref{abstract_weak}.
\begin{theorem}[Wellposedness of the HDM for $\bH^{-2}(\Omega)$ data] \label{thm:wellposedness}
	Let the assumptions $\assum{1}-\assum{4}$ hold. Let $\disc$ be an HD, in the sense of Definition \ref{HD} and let $E_\disc$ in $\assum{5}$ such that  $\Gamma(E_{\disc})$ is bounded. Then, there exists at least one weak solution $\Psi_{\disc}\in \bfXd$ to \eqref{abstract_HS}. 

\smallskip

Furthermore, assume $\disc$ is coercive in the sense of Definition \ref{def:coercive}. If the data satisfies the smallness condition
\begin{equation}\label{assum.f.wellposedness}
\|\cL\|_{-2} < \frac{\overline{\alpha}^2}{\|\cB\| C_\disc^2\Gamma(E_\disc) C_{2,\rm eq}},
\end{equation}
then the discrete solution is unique.
\end{theorem}
\begin{proof}
For a given $\overline{\Psi}_\disc \in \bfXd$, consider the problem that seeks ${\Psi_\disc} \in \bfXd$ be such that, for all $\Phi_\disc \in \bfXd$,
	\begin{align}
	&\mathcal{A}_{\overline{\Psi}_\disc}(\Psi_\disc,\Phi_\disc):=\mathcal{A}(\hd \Psi_\disc, \hd \Phi_\disc)+\mathcal{B}(\hd \overline{\Psi}_\disc, \nabla_\disc\Psi_\disc,\nabla_\disc\Phi_\disc)=\mathcal{L}(E_\disc\Phi_\disc). \label{abstract_HS2}
	\end{align}
	The bilinearity of $\mathcal{A}(\cdot,\cdot)$ and the trilinearity of $\mathcal{B}(\cdot,\cdot,\cdot)$, combined with a fixed $\overline{\Psi}_\disc \in \bfXd$ ensure that  $\mathcal{A}_{\overline{\Psi}_\disc}(\cdot,\cdot)$ is bilinear on $\bfXd$. The continuity of  $\mathcal{A}(\cdot,\cdot)$ and $\mathcal{B}(\cdot,\cdot,\cdot)$ shows that  $\mathcal{A}_{\overline{\Psi}_\disc}(\cdot,\cdot)$  is continuous. The property $\mathcal{B}(\hd \overline{\Psi}_\disc, \nabla_\disc\Psi_\disc, \nabla_\disc\Psi_\disc)=0$ from $\assum{3}$, together with the coercivity of $\mathcal{A}(\cdot,\cdot)$ from $\assum{1}$, implies the following estimate:
\begin{align}
\mathcal{A}_{\overline{\Psi}_\disc}(\Psi_\disc,\Psi_\disc) = \mathcal{A}(\hd \Psi_\disc, \hd \Psi_\disc) \ge  \overline{\alpha}\norm{\Psi_\disc}{\disc}^2, \label{abstract_HS_coercive}
\end{align}
where $\overline{\alpha}$ denotes the coercivity constant of $\mathcal{A}(\cdot,\cdot)$. Consequently, the bilinear form $\mathcal{A}_{\overline{\Psi}_\disc}(\cdot,\cdot)$ is coercive on $\bfXd$. The finite dimensionality of $\bfXd$ and the linearity of $\mathcal{L}(E_\disc\cdot)$ from $\assum{4}$--$\assum{5}$ show that $\mathcal{L}(E_\disc\cdot)$ is a continuous linear functional on $\bfXd$. The Lax-Milgram Lemma thus implies the existence and uniqueness of the solution $\Psi_\disc$ for \eqref{abstract_HS2}.

\smallskip

	Define the mapping $F: \bfXd\rightarrow \bfXd$ by $F(\overline{\Psi}_\disc)=\Psi_\disc$, where $\Psi_\disc$ solves \eqref{abstract_HS2}. The finite dimensionality of $\bfXd$ ensures the continuity of $F$. Moreover,  \eqref{abstract_HS_coercive},  \eqref{abstract_HS2} with $\Phi_\disc =\Psi_\disc$, and \eqref{defn.C2} imply,
\begin{align*}
\overline{\alpha}\,\|\Psi_{\disc}\|_{\disc}^2 
\le \mathcal{A}_{\overline{\Psi}_{\disc}}(\Psi_{\disc},\Psi_{\disc}) &= \mathcal{L}(E_{\disc}\Psi_{\disc}) \le \|\mathcal{L}\|_{-2}\, \|E_{\disc}\Psi_{\disc}\|_{2} \\
&\le C_{\rm 2,eq}\, \|\mathcal{L}\|_{-2}\,
\|\hessian E_{\disc}\Psi_{\disc}\|\le \Gamma(E_{\disc})\, C_{2,\rm eq}\, \|\mathcal{L}\|_{-2}\,
\|\Psi_{\disc}\|_{\disc}.
\end{align*}
	where $\Gamma(E_\disc)$ follows from \eqref{sup.ED.est.Gamma}. Thus, we obtain the stability estimate
	\be 
	\norm{\Psi_\disc}{\disc} \le\overline{\alpha}^{-1}\Gamma(E_{\disc})\, C_{2,\rm eq} \norm{\mathcal{L}}{-2}=:\mathnormal{R}_\disc.\label{abstract_HS_stability}
	\ee
	This inequality demonstrates that $F$ maps $\bfXd$ into the closed ball $B_{R_\disc}$ of radius $R_\disc$ centered at the origin. Consequently, the Brouwer fixed-point theorem guarantees that $F$ possesses at least one fixed point $\Psi_\disc$ within this ball. In view of \eqref{abstract_HS2}, this fixed point constitutes a solution to \eqref{abstract_HS}.
\medskip

To prove uniqueness under \eqref{assum.f.wellposedness}, let $\Psi_\disc^1$, $\Psi_\disc^2 \in \bfXd$ solve \eqref{abstract_HS} and set $\tPsi_\disc:=\Psi_\disc^1-\Psi_\disc^2$. 
\smallskip

The bilinearity of $\mathcal{A}(\cdot,\cdot)$ and the trilinearity of $\mathcal{B}(\cdot,\cdot,\cdot)$ yield
$$\mathcal{A}( \hd\tPsi_\disc, \hd \Phi_\disc)+\mathcal{B}(\hd \tPsi_\disc, \nabla_\disc\Psi_\disc^1,\nabla_\disc\Phi_\disc)+\mathcal{B}(\hd \Psi_\disc^2, \nabla_\disc\tPsi_\disc,\nabla_\disc\Phi_\disc)=0$$
for all $\Phi_\disc \in \bfXd.$ The choice $\Phi_\disc=\tPsi_\disc$ and the property $\mathcal{B}(\hd \Psi_\disc^2, \nabla_\disc\tPsi_\disc,\nabla_\disc\tPsi_\disc)=0$ from $\assum{3}$ lead to
$$\mathcal{A}( \hd\tPsi_\disc, \hd \tPsi_\disc)+\mathcal{B}(\hd \tPsi_\disc, \nabla_\disc\Psi_\disc^1,\nabla_\disc\tPsi_\disc)=0.$$
Since $\mathcal{A}( \hd\tPsi_\disc, \hd \tPsi_\disc) \ge  \overline{\alpha}\norm{\tPsi_\disc}{\disc}^2$, the definitions of $\norm{\mathcal B}{}$ and $C_\disc$ imply
\begin{equation}\label{ineq.uniqueness}
(\overline{\alpha}-\norm{\mathcal B}{}C_\disc^2\norm{\Psi_\disc^1}{\disc})\norm{\tPsi_\disc}{\disc}^2 \le 0.
\end{equation}
Since $\Psi_\disc^1$ is a solution to \eqref{abstract_HS}, it satisfies the stability result \eqref{abstract_HS_stability} given by
\[\norm{\Psi_\disc^1}{\disc} \le\overline{\alpha}^{-1}\Gamma(E_{\disc})\, C_{2,\rm eq} \norm{\mathcal{L}}{-2}.\]
A combination of this and the smallness assumption \eqref{assum.f.wellposedness} ensure that $\overline{\alpha}-\norm{\mathcal B}{}C_\disc^2\norm{ \Psi_\disc^1}{\disc} >0$. Consequently, \eqref{ineq.uniqueness} shows that $\tPsi_\disc=0$, which proves uniqueness.

\end{proof}
\begin{theorem}[Convergence of the HDM  for $\bH^{-1}(\Omega)$ data] \label{convergence_abstract_non-linear}
	Let the assumptions $\assum{1}-\assum{3}$ hold. Let $\assum{4}$ be restricted to $\bsymb{H}^{-1}(\Omega)$. Let $(\disc_m)_{m \in \N}$ be a sequence of HDs, in the sense of Definition \ref{HD}, that is coercive, consistent, limit-conforming, and compact in the sense of Definition \ref{def:coercive}. Also, let the companion operator $E_\disc$ in $\assum{5}$ such that $\omega(E_{\dm}) \to 0$ as $m \to \infty$. Then, for any $m \in \N$, there exists at least one weak solution $\Psi_{\dm} \in \bfX_{\disc_m,0}$ to \eqref{abstract_HS}, with $\disc=\disc_m$. Moreover, as $m \rightarrow \infty$, there exist a subsequence of $(\disc_m)_{m \in \N}$ (denoted using the same notation $(\disc_m)_{m \in \N}$), and a solution $\Psi \in \bfX$ to the abstract problem \eqref{abstract_weak} such that $\Pi_{\disc_m}\Psi_{\dm} \rightarrow\Psi$ in $\bsymb{L}^2(\O)$, $\nabla_{\disc_m}\Psi_{\dm}\rightarrow\nabla \Psi$ in $\bsymb{L}^4(\O;\R^d)$ and $\hessian_{\disc_m}\Psi_{\dm}\rightarrow\hessian \Psi$ in $\bsymb{L}^2(\O;\symd)$.
\smallskip

Furthermore, if the data satisfies the smallness condition
\begin{equation}\nonumber 
\|\cL\|_{-1} < \frac{\overline{\alpha}^2}{\|\cB\| C_\disc^2(\omega(E_\disc)+|\O|^{1/4}C_\disc)C_{1,\rm eq}},
\end{equation}
then the discrete solution is unique.
\end{theorem}
\begin{proof}
Arguments analogous to those in Theorem~\ref{thm:wellposedness} ensure the existence and uniqueness of the discrete solution $\Psi_\disc \in \bfXd$ to \eqref{abstract_HS}. In this case, the radius $R_\disc$ takes the form:
\be \nonumber
R_\disc := \overline{\alpha}^{-1} C_{1,\rm eq} (\omega(E_\disc) + |\Omega|^{1/4} C_\disc)\norm{\mathcal{L}}{-1}.
\ee
More precisely, \eqref{abstract_HS2}, coercivity of $\cA(\cdot,\cdot)$, $\assum{3}$, \eqref{defn.C1}, triangle inequality with $\nabla_\disc \Psi_\disc$, \eqref{sup.ED.est.omega1}, and \eqref{def.CD} provide
\begin{align*}
\overline{\alpha}\,\|\Psi_{\disc}\|_{\disc}^2 
\le \mathcal{A}_{\overline{\Psi}_{\disc}}(\Psi_{\disc},\Psi_{\disc}) &= \mathcal{L}(E_{\disc}\Psi_{\disc}) \le \|\mathcal{L}\|_{-1}\, \|E_{\disc}\Psi_{\disc}\|_{1} \\
&\le C_{\rm 1,eq}\, \|\mathcal{L}\|_{-1}\,
\|\nabla E_{\disc}\Psi_{\disc}\|\\
&
\le \ C_{\rm 1,eq}\, \|\mathcal{L}\|_{-1}\,
(\|\nabla E_{\disc}\Psi_{\disc}-\nabla_\disc \Psi_\disc\|+\|\nabla_\disc \Psi_\disc\|)\\
&\le  \ C_{\rm 1,eq}\, \|\mathcal{L}\|_{-1}\,
(\omega(E_\disc)\|\Psi_\disc\|_\disc +|\O|^{1/4}\|\nabla_\disc \Psi_\disc\|_{0,4})\\
&\le  \ C_{\rm 1,eq}\, \|\mathcal{L}\|_{-1}\,
(\omega(E_\disc)+|\O|^{1/4} C_\disc)\|\Psi_\disc\|_\disc.
\end{align*}
Consequently, $\|\Psi_\disc\|\le R_\disc$ and wellposedness of the discretisation follows as in Theorem~\ref{thm:wellposedness}.
	
\medskip

From here onwards, let $\Psi_{\dm} \in  \bsymb{X}_{\disc_m,0}$ denote such a solution for $\disc=\disc_m.$ Since $\|\Psi_{\disc_m}\|_{\disc_m} \le R_{\disc_m}$, the sequence $(\Psi_{\disc_m})_{m \in \mathbb{N}}$ is bounded in the discrete norm. By virtue of this bound and Lemma~\ref{regularity_nonlinear}, there exists a subsequence (not relabeled) and a limit $\Psi \in \bfX$ such that the following weak convergences hold as $m \to \infty$:
\begin{align}
    \Pi_{\disc_m} \Psi_{\disc_m} &\rightharpoonup \Psi \; \text{ in } \; \bsymb{L}^2(\Omega), \quad \nabla_{\disc_m} \Psi_{\disc_m} \rightharpoonup \nabla \Psi \;\text{ in }\; \bsymb{L}^4(\Omega;\R^d), \quad  \hessian_{\disc_m} \Psi_{\disc_m} \rightharpoonup \hessian \Psi \;  \text{ in } \;\bsymb{L}^2(\Omega;\symd).\nonumber
\end{align}
Furthermore, the compactness property of the HDM in Definition~\ref{def:coercive} $(iv)$ implies the strong convergence of the reconstruction discrete functions and their discrete gradients:
\begin{equation}\label{eqn.strongcon}
    \Pi_{\disc_m} \Psi_{\disc_m} \to \Psi \text{ in } \bsymb{L}^2(\Omega) \quad \text{and} \quad \nabla_{\disc_m} \Psi_{\disc_m} \to \nabla \Psi \text{ in } \bsymb{L}^4(\Omega;\R^d).
\end{equation}

	\medskip

Let $P_\disc:\bfX \rightarrow \bfXd$ be defined by
		\begin{equation}\label{def:PD}
			P_\disc \Phi:= \argmin_{w_\disc \in \bfXd}\Big(\norm{\Pi_\disc w_\disc-\Phi}{}
		+\norm{\nabla_\disc w_\disc-\nabla\Phi}{0,4}	+\norm{\hd w_\disc-\hessian \Phi}{}\Big).
	\end{equation}
The consistency of the sequence $(\disc_m)_{m \in \mathbb{N}}$ in in Definition~\ref{def:coercive} $(ii)$ implies that, as $m \to \infty$, the interpolant $P_{\disc_m} \Phi$ satisfies
\begin{align*}
  &  \Pi_{\disc_m} P_{\disc_m} \Phi \to \Phi \text{ in } \bsymb{L}^2(\Omega), \quad 
    \nabla_{\disc_m} P_{\disc_m} \Phi \to \nabla \Phi \text{ in } \bsymb{L}^4(\Omega;\mathbb{R}^d), \quad \hessian_{\disc_m} P_{\disc_m} \Phi \to \hessian \Phi \text{ in }\bsymb{L}^2(\Omega;\symd).
\end{align*}
Consequently, Lemma~\ref{abstract_trilinear_cv} alongside the bilinearity and continuity of $\mathcal{A}$ provide, as $m \to \infty$,
\begin{multline}
    \mathcal{A}(\hessian_{\disc_m} \Psi_{\disc_m}, \hessian_{\disc_m} P_{\disc_m} \Phi) + B_{\disc}(\hessian_{\disc_m} \Psi_{\disc_m}, \nabla_{\disc_m} \Psi_{\disc_m}, \nabla_{\disc_m} P_{\disc_m} \Phi) \\
    \to \mathcal{A}(\hessian \Psi, \hessian \Phi) + \mathcal{B}(\hessian \Psi, \nabla \Psi, \nabla \Phi). \label{abstract_passlimit1}
\end{multline}
The definition of $\|\cdot\|_{-1}$, the triangle inequality, the embedding $\bsymb{L}^4(\Omega) \hookrightarrow \bsymb{L}^2(\Omega)$, the convergence $\nabla_{\disc_m} P_{\disc_m} \Phi \to \nabla \Phi$ in $\bsymb{L}^4(\Omega;\mathbb{R}^d)$, and $\omega(E_{\disc_m}) \to 0$ as $m \to \infty$ show
\begin{align}
    |\mathcal{L}(E_{\disc_m} P_{\disc_m} \Phi) - \mathcal{L}(\Phi)| 
    &\lesssim \|\mathcal{L}\|_{-1} \|\nabla E_{\disc_m} P_{\disc_m} \Phi - \nabla \Phi\| \nonumber \\
    &\lesssim \|\mathcal{L}\|_{-1} ( \|\nabla E_{\disc_m} P_{\disc_m} \Phi - \nabla_{\disc_m} P_{\disc_m} \Phi\| + \|\nabla_{\disc_m} P_{\disc_m} \Phi - \nabla \Phi\|) \nonumber \\
    &\lesssim \|\mathcal{L}\|_{-1} ( \omega(E_{\disc_m}) \|P_{\disc_m}\Phi\|_\disc+ \|\nabla_{\disc_m} P_{\disc_m} \Phi - \nabla \Phi\|_{0,4} ) \nonumber \\
    &\to 0 \quad \text{as } m \to \infty. \label{abstract_passlimit2}
\end{align}
The substitution $\Phi_{\disc_m} = P_{\disc_m} \Phi$ into \eqref{abstract_HS} with $\disc = \disc_m$, followed by the application of limits \eqref{abstract_passlimit1} and \eqref{abstract_passlimit2}, establishes that $\Psi$ satisfies the weak formulation \eqref{abstract_weak}.
	
	\medskip
	
It remains to prove the strong convergence of $\hessian_{\disc_m}\Psi_{\dm}$. The assumption $\assum{3}$, \eqref{abstract_HS} with $(\disc,\Phi_\disc)= (\disc_m,\Psi_{\dm})$, and  \eqref{abstract_weak} with $\Phi=\Psi$ imply
\begin{equation}\label{abstract_passlimit2.1}
\mathcal{A}(\hd \Psi_{\dm},\hd \Psi_{\dm})-\mathcal{A}(\hessian \Psi,\hessian \Psi)=\mathcal{L}(E_\disc\Psi_{\dm}) -\mathcal{L}(\Psi).
\end{equation}
Arguments analogue to  \eqref{abstract_passlimit2} together with the strong convergence of $\nabla_\disc\Psi_{\dm}$  in \eqref{eqn.strongcon} lead to
\begin{align}
|\cL(E_{\disc_m}\Psi_{\dm})-\cL(\Psi)|&\lesssim\|\cL\|_{-1}(\omega(E_{\dm})\|\Psi_{\disc_m}\|_\disc+\|\nabla_{\dm}\Psi_{\dm}-\nabla \Psi\|_{0,4})\to 0 \mbox{ as }m \to \infty.\nonumber
\end{align}
Consequently, \eqref{abstract_passlimit2.1} results in
	\begin{align*}
	\lim\limits_{m \rightarrow \infty}\mathcal{A}(\hd \Psi_{\dm},\hd \Psi_{\dm})	
	=\mathcal{A}(\hessian \Psi,\hessian \Psi).
	\end{align*}
The coercivity and bilinearity of $\mathcal{A}$, combined with the weak convergence of $\hessian_{\disc_m} \Psi_{\disc_m}$, lead to
\begin{align*}
\limsup_{m \to \infty} \overline{\alpha} \norm{\hessian_{\disc_m} \Psi_{\disc_m} - \hessian \Psi}{}^2 \le \limsup_{m \to \infty} \mathcal{A}(\hessian_{\disc_m} \Psi_{\disc_m} - \hessian \Psi, \hessian_{\disc_m} \Psi_{\disc_m} - \hessian \Psi) = 0.
\end{align*}
This result establishes the strong convergence $\norm{\hessian_{\disc_m} \Psi_{\disc_m} - \hessian \Psi}{} \to 0$ as $m \to \infty$.
\end{proof}

\section{Error estimates}\label{sec:error}

This section discusses the results that enable proofs of uniqueness and error estimates for the HS. These results hold under the assumption that the regular solution (see Definition \ref{def.regular} below) possesses sufficient smoothness and that the data resides in $\bH^{-1}(\Omega)$. The proof is a generalization of ideas from \cite{ng2} to the HDM framework and is provided for completeness. The main results are stated first, with proofs provided in Sections~\ref{sec.wellposedness} and~\ref{sec.existence}. Section~\ref{sec.Newton} establishes the convergence of Newton's method.

\medskip

In this section, we assume that $\mathcal{A}(\cdot,\cdot)$ is the standard $L^2$ inner product on tensors:
\be \label{A.tensor}
\fl \Phi, \, \Theta \in \bsymb{L}^2(\O;\symd)\,,\quad \mathcal{A}(\Phi,\Theta)=\int_{\O}\Phi:\Theta \d\x.
\ee
Similarly, the form $\cB(\cdot,\cdot,\cdot)$ is defined by: for all  $ \Phi \in \bsymb{L}^2(\O;\symd)$ and for all $\Theta, \Xi \in \bsymb{L}^4(\O;\R^d),$
\be\label{B.tensor}
  \mathcal{B}(\Phi,\Theta,\Xi)=\int_{\O}\Phi:h(\Theta, \Xi)\d\x, 
\ee
where $h(\cdot,\cdot)$ is bilinear on $\R^{dk} \times \R^{dk}$.

\medskip

Furthermore, assume that $\Omega \subsetneq \R^d$, $d \le 3$, is a convex domain and the exact solution $\Psi \in \bfX$ to \eqref{abstract_weak} belongs to $\bsymb{H}^3(\O)$. Note that, by Sobolev embeddings, this smoothness implies $\nabla\Psi\in \bsymb{L}^\infty(\O;\R^d)$ and $\Psi\in \bsymb{W}^{2,4}(\O)$. 
\begin{remark}
It is straightforward to verify that the bilinear and trilinear forms associated with both the Navier–Stokes and von  K\`arman equations (detailed in Section \ref{sec:applications}) are consistent with the definitions provided in \eqref{A.tensor} and \eqref{B.tensor}, respectively. 
\smallskip

It is well-known that \cite[Theorem 7]{HBRR} for the Navier–Stokes and von  K\`arman equations, if $\O$ is a convex polygonal domain and the load function belongs to $\bsymb{H}^{-1}(\O)$, then the exact solution $\Psi$ belongs to $\bsymb{H}^3(\O) \cap \bfX$.  
\end{remark}

To establish error estimates for the semi-linear problem, we introduce an auxiliary limit-conformity measure defined by: for all $\xi \in \bH_{\div}(\O):=\{\phi \in L^2(\O;\R^{d \times d}):\, \mbox{div}\phi \in L^2(\O;\R^d)\}$, 
\be\label{def.WDtilde}
\begin{aligned}
	&\widetilde{W}_\disc(\xi):=\max_{w_\disc\in X_{\disc,0}\backslash\{0\}}
	\frac{1}{\norm{w_\disc}{\disc}}\Bigg|\int_\O \Big( \xi:\hd w_\disc+ (\mbox{div}\xi)\cdot\nabla_\disc w_\disc \Big)\d\x \Bigg|.
\end{aligned}
\ee
Note that $\widetilde{W}_\disc$ measures the error in the discrete integration-by-parts formula between the reconstructed Hessian and the reconstructed gradient.

By applying the triangle inequality, \eqref{def.WD}, and \eqref{def.WD_gradient}, it follows naturally that
$$\widetilde{W}_\disc(\xi) \le W_\disc(\xi)+\hat{W}_\disc(\rm{div}\xi).$$


\smallskip

Fixing $\Psi\in \bfX$, a linearization of \eqref{abstract_weak} around $\Psi$ in the direction of $\Theta$ is given by $${\mathbb{A}}_\Psi(\Theta, \Phi):=\mathcal{A}(\hessian \Theta,\hessian \Phi)+\mathcal{B}(\hessian \Psi, \nabla \Theta, \nabla \Phi)+\mathcal{B}(\hessian \Theta, \nabla \Psi, \nabla \Phi) \quad \fl \Phi \in \bfX.$$

\begin{definition}[Regular/non-singular/isolated solution \cite{Brezzi}]\label{def.regular}
The exact solution $\Psi$ to \eqref{abstract_weak} is said to be regular if the linearized problem is well-posed; that is,  for a given $G \in \bsymb{H}^{-1}(\O)$, the problem
\be \label{linearised.problem}
{\mathbb{A}}_\Psi(\Theta, \Phi)=\langle G, \Phi \rangle_{-1,1}\quad \fl\Phi \in \bfX
\ee  
has a unique solution $\Theta \in \bfX$, and this solution satisfies $\norm{\Theta}{2}\lesssim\norm{G}{-1}$. Equivalently, $\Psi$ is a regular solution if and only if there exists a constant $\beta>0$ such that
\be \label{beta}
\beta \le \inf_{ \substack{\Theta \in \bfX \\ |\Theta|_2 = 1} } \,
    \sup_{ \substack{\Phi \in \bfX\\ |\Phi|_2 = 1} } 
    \mathbb{A}_\Psi(\Theta, \Phi).
\ee 
\end{definition}

The HS that corresponds to the linearized problem \eqref{linearised.problem} seeks $\Theta_\disc \in \bfXd$ such that
\be \label{linearised.HS}
{\mathbb{A}}_{\disc, \Psi}( \Theta_\disc,  \Phi_\disc)=\langle G, E_\disc\Phi_\disc\rangle_{-1,1}\quad \fl \Phi_\disc \in \bfXd,
\ee 
where
\begin{align}
& {\mathbb{A}}_{\disc, \Psi}(\Theta_\disc, \Phi_\disc)={}\mathcal{A}(\hd \Theta_\disc,\hd \Phi_\disc)+\mathcal{B}(\hessian \Psi, \nabla_\disc \Theta_\disc, \nabla_\disc \Phi_\disc)+\mathcal{B}(\hd \Theta_\disc, \nabla \Psi, \nabla_\disc \Phi_\disc).\label{discrete.linear}
\end{align}
The proof of well-posedness for the discrete linearized problem \eqref{linearised.HS}  (see Theorem \ref{thm:discrete.lin} below) relies on the coercivity, consistency, and limit-conformity of the HDM, along with the companion operator $E_\disc$ in $\assum{5}$ and its associated estimates \eqref{sup.ED.est.omega1} and \eqref{sup.ED.est.Gamma}.

\smallskip

 Inspired by the notion of space size for gradient discretisations \cite[Definition 2.22]{gdm}, we set
	\begin{align}
		\alpha_\disc:=\sup_{\phi  \in  ({H^{3}(\O)}\cap X)\setminus{\{0\}}}\frac{S_\disc(\phi)}{\norm{\phi}{3}}\quad\mbox{ and }\quad
		\gamma_\disc:=\sup_{\xi \in H^1(\O;\R^{d \times d})\setminus{\{0\}}}\frac{\widetilde{W}_\disc(\xi)}{\norm{\xi}{1}}.\label{WDtilde.sup}
	\end{align}
 \begin{remark}
 	Based on the estimates that can be established on $S_\disc$ and $\widetilde{W}_\disc$, it is expected that $\alpha_\disc$ and $\gamma_\disc$ will be small for HDs based on small meshes, see for example Lemma \ref{th.MorleyAdini}.
 \end{remark}

	\begin{theorem}[Well-posedness of the discrete linearized problem]\label{thm:discrete.lin}
		Let $\Psi \in \bH^3(\O)\cap\bfX$ be a regular solution to \eqref{abstract_weak}. For any $\Gamma\ge 0$, there exists $\rho >0$ such that, if $$C_\disc\le \Gamma,\,\Gamma(E_\disc) \le \Gamma,\,\omega(E_\disc)\le \rho,\,\alpha_\disc \le \rho, \;\mbox{ and }\;\gamma_\disc \le \rho,$$ then the discrete linearized problem \eqref{linearised.HS} is well-posed. 
	\end{theorem}
Since the discrete linearized problem \eqref{linearised.HS} is well-posed, there exists a constant $\hat{\beta}>0$ such that the following discrete inf-sup condition holds.
\begin{equation}\label{def.betabar}
\hat{\beta}	\le \inf_{ \substack{\Theta_\disc \in \bfXd \\ \|\Theta_\disc\|_\disc = 1} } \,
    \sup_{ \substack{\Phi_\disc  \in \bfXd\\ \|\Phi_\disc\|_\disc = 1} } {\mathbb{A}}_{\disc,\Psi}(\Theta_\disc,\Phi_\disc),
\end{equation}
where $\hat{\beta}$ depends on $\Gamma$ and the continuous inf-sup constant $\beta$ in \eqref{beta}.  

	\begin{theorem}[Existence of discrete solution and error estimates]\label{thm.err}
		Let $\Psi \in \bH^3(\O)\cap\bfX$ be a regular solution to \eqref{abstract_weak}. Under the assumptions of Theorem \ref{thm:discrete.lin} and $\delta(E_\disc) \le \rho$, there exists a solution $\Psi_\disc \in \bfXd$ to \eqref{abstract_HS} that satisfies $\norm{ \Psi_\disc-P_\disc \Psi}{\disc}\lesssim \rho$, as well as the following estimates:
		\begin{align}
			\norm{\hessian \Psi-\hd\Psi_\disc}{}&\lesssim \rho,\ 
			\norm{\nabla \Psi-\nabla_\disc\Psi_\disc}{} \lesssim \rho,\,\mbox{ and }\,
			\norm{ \Psi-\Pi_\disc\Psi_\disc}{}\lesssim \rho,\label{est.func}
		\end{align}	
where the hidden constants in ``$\lesssim$'' depend on $\hat{\beta}$, $\Gamma$, $\Psi$ but not $\rho$ or $\disc$.
	\end{theorem}		
\subsection{Well-posedness of the discrete linearised problem}\label{sec.wellposedness}
The following lemma provides essential stability and regularity estimates for the dual problem. These results are instrumental in establishing the well-posedness of the discrete system \eqref{linearised.HS}.

\begin{lemma}[Well-posedness of the dual problem]
	Let $\Psi \in \bH^3(\O)\cap\bfX$ be a regular solution to \eqref{abstract_weak}. Then the dual problem defined by: given $Q \in \bsymb{H}^{-1}(\O)$, find $\zeta \in \bfX$ such that
	\be \label{linearised.dual.problem}
	{\mathbb{A}}_\Psi( \Phi,\zeta)=\langle Q, \Phi \rangle_{-1,1} \quad \fl  \Phi \in \bfX,
	\ee 
	is well-posed and satisfies the \emph{a priori} bound:
	\be \label{dual.est}
	\norm{\zeta}{2}\lesssim\norm{Q}{-1}, \quad \norm{\zeta}{3}\lesssim\norm{Q}{-1}.
	\ee
\end{lemma} 
\begin{proof}
The assumption that $\Psi$ is regular solution ensures that the linearised bilinear form  ${\mathbb{A}}_\Psi(\cdot,\cdot)$ satisfies the inf-sup condition \eqref{beta}. The existence and uniqueness of $\zeta$ then follow from the Babuska theorem. The a priori estimates \eqref{dual.est} follow from standard elliptic regularity results \cite{HBRR}.
\end{proof}
\begin{proof}[Proof of Theorem \ref{thm:discrete.lin}]
	Since $\bfXd$ is finite dimensional and \eqref{linearised.HS} is linear, the existence of an \emph{a priori} bound implies that the problem has a unique solution. For $\Phi_\disc \in \bfXd$, the definitions \eqref{discrete.linear}, \eqref{A.tensor} and \eqref{B.tensor} of $\mathbb{A}_{\disc,\Psi}$, $\mathcal A$ and $\mathcal B$, the generalised H\"older inequality, and the definition \eqref{def.CD} of $C_\disc$ leads to the following G{\aa}rdings-type inequality.
	\begin{align}
	{\mathbb{A}}_{\disc, \Psi}( \Phi_\disc,  \Phi_\disc)={}&\mathcal{A}(\hd \Phi_\disc,\hd \Phi_\disc)+\mathcal{B}(\hessian \Psi, \nabla_\disc \Phi_\disc, \nabla_\disc \Phi_\disc)+\mathcal{B}(\hd \Phi_\disc, \nabla \Psi, \nabla_\disc \Phi_\disc)\nonumber\\
	\gtrsim {}& \norm{\Phi_\disc}{\disc}^2-C_\disc\norm{\Phi_\disc}{\disc}\norm{\nabla_\disc \Phi_\disc}{}\norm{\hessian\Psi}{0,4}-\norm{ \Phi_\disc}{\disc}\norm{\nabla_\disc \Phi_\disc}{}\norm{\nabla\Psi}{0,\infty},\nonumber
	\end{align} 
	where the constant suppressed in $\gtrsim$ is independent of $\disc$. This, the choice $\Phi_\disc=\Theta_\disc$ in \eqref{linearised.HS}, triangle inequality with $\nabla_\disc \Theta_\disc$, \eqref{sup.ED.est.omega1}, $\bL^4(\Omega) \hookrightarrow \bL^2(\Omega)$, and \eqref{def.CD} show
	\be\label{ineq.HD} 
	\norm{\Theta_\disc}{\disc}\lesssim \big(C_\disc\norm{\hessian\Psi}{0,4}+\norm{\nabla\Psi}{0,\infty}\big)\norm{\nabla_\disc\Theta_\disc}{}+(\omega(E_\disc)+C_\disc)\norm{G}{-1}.
	\ee
The triangle inequality and \eqref{sup.ED.est.omega1} provide an estimate for  $\norm{\nabla_\disc \Theta_\disc}{}$ in the above expression as
	\begin{align}
	\norm{\nabla_\disc \Theta_\disc}{} \le{}& \norm{\nabla_\disc \Theta_\disc-\nabla E_\disc \Theta_\disc}{}+\norm{\nabla E_\disc \Theta_\disc}{}\le  \omega(E_\disc) \norm{\Theta_\disc}{\disc}+\norm{\nabla E_\disc \Theta_\disc}{}.
	\label{ineq.nablaED}
	\end{align}
	To estimate $\norm{\nabla E_\disc \Theta_\disc}{}$, choose $Q=-\Delta E_\disc\Theta_\disc$ and $\Phi=E_\disc \Theta_\disc$ in  \eqref{linearised.dual.problem}. An introduction of the terms $\pm\mathcal{B}(\hessian \Psi, \nabla E_\disc \Theta_\disc, \nabla_\disc P_\disc\zeta)$ and $\pm\mathcal{B}(\hessian E_\disc \Theta_\disc, \nabla \Psi, \nabla_\disc P_\disc\zeta),$ and \eqref{linearised.HS} with $\Phi_\disc=P_\disc \zeta$ yield
	\begin{align*}
	\norm{\nabla E_\disc \Theta_\disc}{}^2
	={}&{\mathbb{A}}_\Psi(E_\disc \Theta_\disc,\zeta)
	=\mathcal{A}(\hessian E_\disc \Theta_\disc,\hessian \zeta)+\mathcal{B}(\hessian \Psi, \nabla E_\disc \Theta_\disc, \nabla \zeta-\nabla_\disc P_\disc\zeta)\\
	& \quad +\mathcal{B}(\hessian E_\disc \Theta_\disc, \nabla \Psi, \nabla \zeta-\nabla_\disc P_\disc \zeta) +\mathcal{B}(\hessian \Psi, \nabla E_\disc \Theta_\disc, \nabla_\disc P_\disc\zeta)\\
	& \quad 
	+\mathcal{B}(\hessian E_\disc \Theta_\disc, \nabla \Psi, \nabla_\disc P_\disc\zeta)-{\mathbb{A}}_{\disc, \Psi}( \Theta_\disc, P_\disc\zeta)+\langle G, E_\disc P_\disc \zeta\rangle_{-1,1}.
	\end{align*}
	An introduction of $\pm\mathcal{A}(\hd \Theta_\disc,\hessian \zeta)$ and $\pm\mathcal{B}(\hessian E_\disc \Theta_\disc-\hessian_\disc\Theta_\disc,\nabla \Psi,\nabla\zeta)$, leads to 
	\begin{align}
	\norm{\nabla &E_\disc \Theta_\disc}{}^2	=\mathcal{A}(\hessian E_\disc \Theta_\disc-\hd \Theta_\disc,\hessian \zeta)+\mathcal{A}(\hd \Theta_\disc,\hessian \zeta-\hd P_\disc \zeta)\nonumber\\
	&\quad+\mathcal{B}(\hessian \Psi,\nabla E_\disc \Theta_\disc,\nabla\zeta-\nabla_\disc P_\disc\zeta)+\mathcal{B}(\hessian E_\disc \Theta_\disc,\nabla \Psi,\nabla\zeta-\nabla_\disc P_\disc\zeta)\nonumber\\
	&\quad +\mathcal{B}(\hessian \Psi,\nabla E_\disc \Theta_\disc-\nabla_\disc\Theta_\disc,\nabla_\disc P_\disc\zeta)+\mathcal{B}(\hessian E_\disc \Theta_\disc-\hessian_\disc\Theta_\disc,\nabla \Psi,\nabla\zeta)\nonumber\\
	&\quad+\mathcal{B}(\hessian E_\disc \Theta_\disc-\hessian_\disc\Theta_\disc,\nabla \Psi,\nabla_\disc P_\disc\zeta-\nabla\zeta)+\langle G,E_\disc P_\disc \zeta\rangle_{-1,1}=:\sum_{i=1}^{8}T_i.\label{nablaED}
	\end{align}
	We now estimate each $T_i$ for $i=1,\cdots,8$. The definition of $\mathcal{A}(\cdot,\cdot)$ in \eqref{A.tensor}, an integration by parts, \eqref{def.WDtilde}, Cauchy-Schwarz inequality, \eqref{sup.ED.est.omega1}, and \eqref{WDtilde.sup} imply
	\begin{align}
	T_1&=\mathcal{A}(\hessian E_\disc \Theta_\disc-\hd \Theta_\disc,\hessian \zeta)= \int_\O (\hessian E_\disc \Theta_\disc-\hd \Theta_\disc):\hessian \zeta \dx\nonumber\\
	&\le-\int_\O \nabla E_\disc \Theta_\disc\cdot\mbox{div}(\hessian \zeta)\dx+\int_\O  \nabla_\disc \Theta_\disc\cdot\mbox{div}(\hessian \zeta)\dx+\widetilde{W}_\disc(\hessian \zeta)\norm{\Theta_\disc}{\disc}\nonumber\\
	&\le  \big({\omega(E_\disc) \norm{\mbox{div}(\hessian \zeta)}{}+\gamma_\disc\norm{\hessian \zeta}{1}}{}\big)\norm{\Theta_\disc}{\disc}. \label{t1.est}
	\end{align}
	A use of \eqref{A.tensor}, Cauchy--Schwarz inequality, \eqref{def.SD}, and \eqref{WDtilde.sup} yields
	\be
	T_2=\mathcal{A}(\hd \Theta_\disc,\hessian \zeta-\hd P_\disc \zeta)\le \norm{\Theta_\disc}{\disc}S_\disc(\zeta)\le\alpha_\disc\norm{\Theta_\disc}{\disc} \norm{\zeta}{{3}}. \label{t2.est}
	\ee
	The generalised  H\"older inequality, Sobolev imbedding $H^1(\O)\hookrightarrow L^4(\O)$, \eqref{def.SD}, \eqref{sup.ED.est.Gamma}, and \eqref{WDtilde.sup} reveal
	\begin{align}
	T_3={}\int_\O\hessian \Psi:h(\nabla E_\disc \Theta_\disc,\nabla\zeta-&\nabla_\disc P_\disc\zeta)\d\x
	\lesssim{}\Gamma(E_\disc) \norm{\hessian \Psi}{}\norm{\Theta_\disc}{\disc}S_\disc(\zeta)\nonumber\\
	\lesssim {}&  \alpha_\disc\Gamma(E_\disc) \norm{\hessian \Psi}{}\norm{\Theta_\disc}{\disc}\norm{\zeta}{{3}}, \label{t3.est}
	\end{align}
	\begin{align}
	T_4={}\int_\O \hessian E_\disc \Theta_\disc: h(\nabla \Psi,\nabla\zeta-&\nabla_\disc P_\disc\zeta)\d\x\lesssim{} \Gamma(E_\disc) \norm{ \Theta_\disc}{\disc}S_\disc(\zeta)\norm{\nabla \Psi}{0,4} \nonumber\\
	\lesssim {}&  \alpha_\disc\Gamma(E_\disc) \norm{\Theta_\disc}{\disc}\norm{\Psi}{2}\norm{\zeta}{{3}}, \label{t4.est}
	\end{align}
	and, since $\norm{\nabla_\disc P_\disc\zeta}{0,4}\le S_\disc(\zeta)+\norm{\nabla \zeta}{0,4}$,
	\begin{align}
	T_5&
	=\int_\O\hessian \Psi:h(\nabla E_\disc \Theta_\disc-\nabla_\disc\Theta_\disc,\nabla_\disc P_\disc\zeta)\d\x \nonumber\\&\lesssim  \omega(E_\disc)\norm{\hessian \Psi}{0,4} \norm{ \Theta_\disc}{\disc}\big(S_\disc(\zeta)+\norm{\nabla \zeta}{0,4}\big)\nonumber\\
	&\lesssim \omega(E_\disc)\big(\alpha_\disc+1\big)\norm{ \Psi}{3} \norm{ \Theta_\disc}{\disc} \norm{ \zeta}{{3}}. \label{t5.est}
	\end{align}	
	A use of \eqref{B.tensor}, integration by parts, \eqref{def.WDtilde},  Cauchy-Schwarz inequality, \eqref{sup.ED.est.omega1}, and \eqref{WDtilde.sup} leads to
	\begin{align}
	T_6
	&=\int_{\O}\left(\hessian E_\disc \Theta_\disc-\hessian_\disc\Theta_\disc\right):h(\nabla \Psi,\nabla\zeta)\dx\nonumber\\
	&\le\int_{\O}(\nabla_\disc \Theta_\disc-\nabla E_\disc \Theta_\disc)\cdot\mbox{div}(h(\nabla \Psi,\nabla\zeta))\dx+\widetilde{W}_\disc(h(\nabla \Psi,\nabla\zeta))\norm{ \Theta_\disc}{\disc}\nonumber\\
	&\le \left(\omega(E_\disc)\norm{\mbox{div}(h(\nabla \Psi,\nabla\zeta))}{}+\widetilde{W}_\disc(h(\nabla \Psi,\nabla\zeta))\right)\norm{\Theta_\disc}{\disc}\nonumber\\
	&\lesssim (\omega(E_\disc)+\gamma_\disc) \norm{ \Theta_\disc}{\disc}\norm{\Psi}{2,4}\norm{ \zeta}{2,4}\lesssim (\omega(E_\disc)+\gamma_\disc) \norm{ \Theta_\disc}{\disc}\norm{\Psi}{3}\norm{ \zeta}{3}.\label{t6.est}
	\end{align}
The generalised  H\"older inequality, \eqref{sup.ED.est.Gamma}, \eqref{def.SD}, \eqref{WDtilde.sup}, and Sobolev imbedding $H^1(\O)\hookrightarrow L^4(\O)$ show
	\begin{align}
	T_7&=\int_\O (\hessian E_\disc \Theta_\disc-\hessian_\disc\Theta_\disc):h(\nabla \Psi,\nabla_\disc P_\disc\zeta-\nabla\zeta)\d\x\nonumber\\& \lesssim (\Gamma(E_\disc)+1)\norm{ \Theta_\disc}{\disc}S_\disc(\zeta)\norm{\nabla \Psi}{0,4} \lesssim  \alpha_\disc(\Gamma(E_\disc)+1)\norm{ \Theta_\disc}{\disc}\norm{\Psi}{2}\norm{\zeta}{{3}}. \label{t7.est}
	\end{align}
The definition of $\|\cdot \|_{-1}$, triangle inequalities, \eqref{def.SD}, \eqref{sup.ED.est.omega1}, and \eqref{WDtilde.sup} imply
	\begin{align}
	T_8&=\langle G,E_\disc P_\disc \zeta\rangle \lesssim \|G\|_{-1}\|\nabla E_\disc P_\disc \zeta\|\nonumber \\
&\le  \|G\|_{-1}(\|\nabla E_\disc P_\disc \zeta-\nabla_\disc P_\disc \zeta\|+\|\nabla_\disc P_\disc \zeta-\nabla \zeta\|+\|\nabla \zeta\|)\nonumber\\
&\lesssim  \|G\|_{-1}\left(\omega(E_\disc)\|P_\disc \zeta\|_\disc+S_\disc(\zeta)+\|\nabla \zeta\|\right)\nonumber\\
&\le  \|G\|_{-1}\left(\omega(E_\disc)(\|\hessian_\disc P_\disc \zeta-\hessian \zeta\|+\|\hessian \zeta\|)+S_\disc(\zeta)+\|\nabla \zeta\|\right)\nonumber\\
&\le  \norm{G}{-1}\left( \omega(E_\disc)(\alpha_\disc+1)+\alpha_\disc+1\right)\|\zeta\|_3.\label{t8.est}
	\end{align}
A substitution of \eqref{t1.est}--\eqref{t8.est} into \eqref{nablaED}, and a use of \eqref{dual.est} and $\norm{Q}{-1} \le \norm{\nabla E_\disc \Theta_\disc}{}$ shows
	\begin{equation}
	\begin{aligned}
	\norm{\nabla E_\disc \Theta_\disc}{}{} \lesssim & \norm{ \Theta_\disc}{\disc}\big(\omega(E_\disc)+\gamma_\disc+\alpha_\disc(1+\Gamma(E_\disc)+\omega(E_\disc))\big) \nonumber\\
&\qquad +\norm{G}{-1}\left( \omega(E_\disc)(\alpha_\disc+1)+\alpha_\disc+1\right),\nonumber 
	\end{aligned}
	\end{equation}	
	where the constant in $\lesssim$ is independent of $\disc$, but depends on $\Psi$. A combination of this, \eqref{ineq.HD}, and \eqref{ineq.nablaED} yields a positive constant $C$ independent of $\disc$ such that
	\begin{equation}\label{Est.HD.Theta.final}
	\begin{aligned}
	\norm{\Theta_\disc}{\disc}\le{}& C(C_\disc+1)\big[\omega(E_\disc)+\gamma_\disc+\alpha_\disc(1+\Gamma(E_\disc)+\omega(E_\disc))\big]\norm{\Theta_\disc}{\disc}\\
	&\quad  \quad +C\left( \omega(E_\disc)(\alpha_\disc+1)+\alpha_\disc+1+C_\disc \right)\norm{G}{-1}.
	\end{aligned}
	\end{equation}
	For $C_\disc\le \Gamma$ and $\Gamma(E_\disc) \le \Gamma$, choose $\rho$ such that 
	$$C\rho(\Gamma+1)(3+\Gamma+\rho) \le 1/2.$$
	If   $\omega(E_\disc)\le \rho$, $\alpha_\disc \le \rho$, and $\gamma_\disc \le \rho$, then the estimate \eqref{Est.HD.Theta.final} gives $\norm{\Theta_\disc}{\disc}\le2C(\rho(\rho+2)+1+\Gamma)\norm{G}{-1}$, which is the sought {a priori} estimate on the solution to the linearized problem \eqref{linearised.HS}.
\end{proof}

\subsection{Existence and error estimates}\label{sec.existence}
The following lemma establishes that the perturbed linearized bilinear form constructed using $P_\disc(\cdot)$ in \eqref{def:PD} defines an isomorphism provided that $\Psi$ is a regular solution of the continuous problem.

\begin{lemma}[Discrete inf-sup condition of perturbed bilinear form] \label{lemma.discrete.pert}
	Under the assumptions of Theorem \ref{thm:discrete.lin}, and $\rho$ further satisfies $\displaystyle \rho \norm{\Psi}{{3}}\le \frac{\hat{\beta}}{2\|\cB\| C_\disc(1+C_\disc)}$, the perturbed bilinear form defined by
	\begin{align}
	\mathfrak{A}_{\disc,\Psi}(\Theta_\disc, \Phi_\disc)={}&\mathcal{A}(\hd \Theta_\disc,\hd \Phi_\disc)+\mathcal{B}(\hd P_\disc \Psi, \nabla_\disc \Theta_\disc, \nabla_\disc \Phi_\disc)\nonumber \\
	&\qquad +\mathcal{B}(\hd \Theta_\disc, \nabla_\disc P_\disc \Psi, \nabla_\disc \Phi_\disc)\label{perturbed.form}
	\end{align}
satisifies a discrete inf-sup condition on $\bfXd \times \bfXd$.
\end{lemma}
\begin{proof}
	Let $\Theta_\disc \in \bfXd$. Theorem \ref{thm:discrete.lin} shows the existence of some $\Phi_\disc \in \bfXd$ such that $\norm{\Phi_\disc}{\disc}=1$ and $\hat{\beta} \norm{\Theta_\disc}{\disc} \le {\mathbb{A}}_{\disc, \Psi}(\Theta_\disc,\Phi_\disc)$. This, \eqref{perturbed.form}, \eqref{discrete.linear}, generalised  H\"older inequality, \eqref{def.CD}, and \eqref{def.SD} show
	\begin{align*}
	\mathfrak{A}_{\disc,\Psi}(\Theta_\disc, \Phi_\disc)
	&={\mathbb{A}}_{\disc, \Psi}(\Theta_\disc,\Phi_\disc)-\mathcal{B}(\hessian \Psi-\hd P_\disc \Psi, \nabla_\disc \Theta_\disc, \nabla_\disc \Phi_\disc)\\
	&\qquad \qquad\quad-\mathcal{B}(\hd \Theta_\disc, \nabla \Psi -\nabla_\disc P_\disc \Psi, \nabla_\disc \Phi_\disc)\\
	&\ge \hat{\beta}\norm{ \Theta_\disc}{\disc}-\|\cB\| C_\disc(1+C_\disc)S_\disc(\Psi)\norm{ \Theta_\disc}{\disc}\ge \frac{\hat{\beta}}{2}\norm{ \Theta_\disc}{\disc},
	\end{align*}
	provided $\displaystyle S_\disc(\Psi)\le \frac{\hat{\beta}}{2\|\cB\| C_\disc(1+C_\disc)}.$ The definition of $\alpha_\disc$ in \eqref{WDtilde.sup} yields $S_\disc(\Psi) \le \alpha_\disc \|\Psi\|_3$. Hence the required result follows, provided that $\rho$ is as in Theorem \ref{thm:discrete.lin} and further satisfies $\displaystyle \rho \norm{\Psi}{{3}}\le \frac{\hat{\beta}}{2\|\cB\| C_\disc(1+C_\disc)}.$
\end{proof}

\medskip

Under the assumptions of Lemma~\ref{lemma.discrete.pert}, define the nonlinear operator $\mu:\bfXd \rightarrow \bfXd$ such that, for $\Theta_\disc\in \bfXd$, $\mu(\Theta_\disc)$ is the unique solution to:
\begin{align}
\mathfrak{A}_{\disc,\Psi}(\mu(\Theta_\disc),&\Phi_\disc)=\mathcal{L}(E_\disc \Phi_\disc)+\mathcal{B}(\hd P_\disc \Psi, \nabla_\disc \Theta_\disc, \nabla_\disc \Phi_\disc)\nonumber \\
&+\mathcal{B}(\hd \Theta_\disc, \nabla_\disc P_\disc \Psi, \nabla_\disc \Phi_\disc)-\mathcal{B}(\hd \Theta_\disc, \nabla_\disc \Theta_\disc, \nabla_\disc \Phi_\disc), \label{def.mu}
\end{align}
for all $\Phi_\disc \in \bfXd$. Observe that any fixed point of $\mu$ is a solution to \eqref{abstract_HS} and that the converse also holds true. Hence, in order to establish the existence of a solution to \eqref{abstract_HS}, we will prove that the mapping $\mu$ has a fixed point. For $R >0$, define
$$B_R(P_\disc \Psi)=\{\Phi_\disc \in \bfXd: \norm{ \Phi_\disc-P_\disc \Psi}{\disc}\le R\}.$$

\begin{theorem}(Mapping of ball into ball) \label{thm.map}
	Let $\Psi \in \bH^3(\O)\cap\bfX$ be  a regular solution to \eqref{abstract_weak}.  Under the assumptions of Lemma~\ref{lemma.discrete.pert} with sufficiently small choice of $\rho$, there exists $\mathcal K>0$ not depending on $\rho$ or $\disc$ such that, setting $R=\mathcal{K}\rho$, $\mu$ maps $B_R(P_\disc\Psi)$ into itself. 
\end{theorem}
\begin{proof}
Since $\mu(\Theta_\disc)- P_\disc \Psi \in \bfXd$, the discrete inf-sup condition of $\mathfrak{A}_{\disc,\Psi}(\cdot,\cdot)$ in Lemma~\ref{lemma.discrete.pert} shows that there exists $\Phi_\disc \in \bfXd$ such that $\norm{ \Phi_\disc}{\disc}=1$ and
	\be \label{non.sing}
	\frac{\hat{\beta}}{2}\norm{ \mu(\Theta_\disc)- P_\disc \Psi}{\disc} \le \mathfrak{A}_{\disc,\Psi}(\mu(\Theta_\disc)-P_\disc \Psi, \Phi_\disc).
	\ee
	A use of \eqref{def.mu}, \eqref{perturbed.form}, and \eqref{abstract_weak} with $\Phi=E_\disc \Phi_\disc$ yields
	\begin{align}
	\mathfrak{A}_{\disc,\Psi}{}(\mu(\Theta_\disc)-P_\disc \Psi, \Phi_\disc)={}&\mathcal{L}(E_\disc \Phi_\disc)+\mathcal{B}(\hd P_\disc \Psi, \nabla_\disc \Theta_\disc, \nabla_\disc \Phi_\disc)\nonumber\\
	&+\mathcal{B}(\hd \Theta_\disc, \nabla_\disc P_\disc \Psi, \nabla_\disc \Phi_\disc)-\mathcal{B}(\hd \Theta_\disc, \nabla_\disc \Theta_\disc, \nabla_\disc \Phi_\disc)\nonumber\\
	&-\mathcal{A}(\hd P_\disc \Psi,\hd \Phi_\disc)-2\mathcal{B}(\hd P_\disc \Psi, \nabla_\disc P_\disc\Psi, \nabla_\disc \Phi_\disc)\nonumber\\
	={}&\left[\mathcal{A}(\hessian \Psi,\hessian E_\disc \Phi_\disc)-\mathcal{A}(\hd P_\disc \Psi,\hd \Phi_\disc)\right]\nonumber\\
	&+\left[\mathcal{B}(\hessian \Psi, \nabla \Psi, \nabla E_\disc \Phi_\disc)-\mathcal{B}(\hd P_\disc \Psi, \nabla_\disc P_\disc\Psi, \nabla_\disc \Phi_\disc)\right]\nonumber\\
	&+\mathcal{B}(\hd P_\disc \Psi-\hd \Theta_\disc, \nabla_\disc \Theta_\disc-\nabla_\disc P_\disc\Psi, \nabla_\disc \Phi_\disc)=:\sum_{i=1}^{3}B_i.\label{est.b1b4}
	\end{align}
An introduction of the term $\hd \Phi_\disc$ and \eqref{t1.est} with the pair $(\Theta_\disc,\zeta)$ replaced by $(\Phi_\disc,\Psi)$, \eqref{def.SD}, and \eqref{WDtilde.sup} lead to
	\begin{align}
	B_1
	&\le\big|\mathcal{A}(\hessian \Psi,\hessian E_\disc \Phi_\disc-\hd \Phi_\disc)\big|+\big|\mathcal{A}(\hessian \Psi-\hd P_\disc \Psi,\hd \Phi_\disc)\big|\nonumber\\
	&	\le  \big(\omega(E_\disc)\norm{\mbox{div}(\hessian \Psi)}{}+\gamma_\disc\norm{\hessian \Psi}{1}+\alpha_\disc\norm{\Psi}{{3}}\big)\norm{\Phi_\disc}{\disc}.\label{est.b2}
	\end{align}
The continuity of $\mathcal{B}(\cdot,\cdot,\cdot)$, \eqref{sup.ED.est.omega1}, \eqref{def.CD}, \eqref{def.SD}, the Sobolev imbedding $H^1(\O)\hookrightarrow L^4(\O)$, and \eqref{WDtilde.sup} provide
	\begin{align}
	B_2&\le \big|\mathcal{B}(\hessian \Psi, \nabla \Psi, \nabla E_\disc \Phi_\disc-\nabla_\disc \Phi_\disc)-\mathcal{B}(\hd P_\disc \Psi-\hessian \Psi, \nabla_\disc P_\disc\Psi, \nabla_\disc \Phi_\disc)\big|\nonumber\\
	&\qquad \qquad + \big|\mathcal{B}(\hessian \Psi, \nabla \Psi-\nabla_\disc P_\disc\Psi, \nabla_\disc \Phi_\disc)\big| \nonumber\\
	&\lesssim  \big[\omega(E_\disc)\norm{\hessian \Psi}{0,4}\norm{\nabla \Psi}{0,4}+\big(\norm{\nabla_\disc P_\disc\Psi}{0,4}+ \norm{\hessian\Psi}{}\big)C_\disc S_\disc(\Psi)\big]\norm{\Phi_\disc}{\disc}\nonumber\\
	&\lesssim  \big[\omega(E_\disc)\norm{\Psi}{3}^2+\big(\alpha_\disc\norm{\Psi}{{3}}+ \norm{\hessian\Psi}{}\big)C_\disc\alpha_\disc\norm{\Psi}{{3}}\big]\norm{\Phi_\disc}{\disc}.\label{est.b3}
	\end{align}
	Finally, the continuity of $\cB(\cdot,\cdot,\cdot)$ and \eqref{def.CD} imply
	\begin{align}
	B_3&\lesssim  C_\disc^2\norm{\Theta_\disc-P_\disc \Psi}{\disc}^2\norm{ \Phi_\disc}{\disc}.\label{est.b4}
	\end{align}
	A substitution of \eqref{est.b2}-\eqref{est.b4} in \eqref{est.b1b4} (together with the fact that $\norm{\Phi_\disc}{\disc}=1$) and then in \eqref{non.sing} shows that
	\begin{align*}
	\norm{\mu(\Theta_\disc)-P_\disc \Psi}{\disc}& \le \mathcal{K}_1\big(\gamma_\disc + \alpha_\disc(C_\disc+1+C_\disc \alpha_\disc)+\omega(E_\disc)  +C_\disc^2\norm{ \Theta_\disc-P_\disc \Psi}{\disc}^2\big),
	\end{align*}
	where $\mathcal{K}_1>0$ is independent of $\disc$,  but depends on $\hat{\beta}$ and $\norm{\Psi}{{3}}$. Given the assumptions of Lemma~\ref{lemma.discrete.pert}, there exists $\Gamma >0$ and $\rho >0$ be such that $C_\disc\le \Gamma$, $\gamma_\disc \le \rho$, $\omega(E_\disc)\le \rho$, and $\alpha_\disc \le \rho$. Then,
	$$\norm{\mu(\Theta_\disc)- P_\disc \Psi}{\disc}\le  \mathcal{K}_1\big(\rho(\Gamma+3+\Gamma\rho)+\Gamma^2\norm{\Theta_\disc-P_\disc \Psi}{\disc}^2\big).$$
	In addition to the choice of $\rho$ in Lemma~\ref{lemma.discrete.pert}, let $\rho>0$ be chosen sufficiently small such that $4\mathcal{K}_1^2\Gamma^2\rho(\Gamma+3+\Gamma\rho)\le 1$. Set $R:=2\mathcal{K}_1\rho(\Gamma+3+\Gamma\rho)$. Then, $\norm{\Theta_\disc- P_\disc \Psi}{\disc} \le R$ implies
	$$	\norm{ \mu(\Theta_\disc)-P_\disc \Psi}{\disc}\le \mathcal{K}_1\rho(\Gamma+3+\Gamma\rho)\big(1+4\mathcal{K}_1^2\Gamma^2\rho(\Gamma+3+\Gamma\rho)\big)\le R.$$
	This completes the proof.
\end{proof}

We can now prove Theorem \ref{thm.err}.	

\begin{proof}[Proof of Theorem \ref{thm.err}]
	
The inf-sup condition on $\mathfrak{A}_{\disc,\Psi}(\cdot,\cdot)$ and the boundedness of the multilinear forms in the finite-dimensional setting show that $\mu$ is continuous. Since $\mu$ maps the ball $B_R(P_\disc \Psi)$ to itself from Theorem \ref{thm.map}, the Brouwer fixed point theorem shows that the mapping $\mu$ has a fixed point in this ball, say $\Psi_\disc$. Hence, $\Psi_\disc$ is a solution of \eqref{abstract_HS} that satisfies $\norm{ \Psi_\disc- P_\disc \Psi}{\disc}\le R$, where $R=\mathcal{K}\rho$. This proves the existence part in Theorem \ref{thm.err}.
	
\medskip

An introduction of $\hd\Psi_\disc$, the definition of $S_\disc$ and \eqref{WDtilde.sup} show that
	\begin{align}
	\norm{\hessian \Psi-\hd\Psi_\disc}{}&\le \norm{\hessian \Psi-\hd P_\disc\Psi}{}+\norm{ P_\disc \Psi-\Psi_\disc}{\disc}\le (\mathcal{K}+1)\rho.\nonumber
	\end{align}
In a similar way, the triangle inequality, \eqref{def.SD}, and \eqref{def.CD} lead to
\begin{align*}
\|\nabla \Psi -\nabla_\disc \Psi_\disc \|&\le \|\nabla \Psi -\nabla_\disc P_\disc \Psi \|+\|\nabla_\disc P_\disc \Psi -\nabla_\disc \Psi_\disc \|\\
&\le |\O|^{1/4}\left(  \|\nabla \Psi -\nabla_\disc P_\disc \Psi \|_{0,4}+\|\nabla_\disc P_\disc \Psi -\nabla_\disc \Psi_\disc \|_{0,4}\right)\\
&\lesssim S_\disc(\Psi)+C_\disc \| P_\disc \Psi - \Psi_\disc \|_{\disc} \lesssim (1+C_\disc \mathcal{K})\rho
\end{align*}
	and
\[\|\Psi -\Pi_\disc \Psi_\disc \|\le \|\Psi -\Pi_\disc P_\disc \Psi \|+\|\Pi_\disc P_\disc \Psi -\Pi_\disc \Psi_\disc \|\lesssim (1+C_\disc \mathcal{K})\rho.\]
\end{proof}
The following lemma establishes the local uniqueness of the solution of \eqref{abstract_HS}.
\begin{lemma}[Local uniqueness]\label{thm.contraction}
	Let $\O$ be a convex domain and $\Psi \in \bH^3(\Omega)\cap \bfX$ be a regular solution to \eqref{abstract_weak}. For $\Theta_\disc^1, \Theta_\disc^2 \in B_R(P_\disc \Psi)$ with $R$ as defined as in Theorem \ref{thm.map}, the following result holds true.
	$$\norm{\mu( \Theta_\disc^1)-\mu(\Theta_\disc^2)}{\disc}\lesssim R\norm{ \Theta_\disc^1-\Theta_\disc^2}{\disc},$$
	where the constant in $\lesssim$ is independent of $\disc$, but depends on $\Gamma$ and $\hat{\beta}$. Hence, if $\rho$ is small enough, $\mu$ is a contraction on $B_R(P_\disc \Psi)$ and the solution to \eqref{abstract_HS} in this ball is unique.
\end{lemma}
\begin{proof}
	For $i=1,2$ and $\Theta_\disc^i \in B_R(P_\disc \Psi)$, let $\mu(\Theta_\disc^i)$ be the solution to
\begin{align}
	\mathfrak{A}_{\disc,\Psi}(\mu(\Theta_\disc^i),\Phi_\disc)={}&	\mathcal{L}(E_\disc \Phi_\disc)+\mathcal{B}(\hd P_\disc \Psi, \nabla_\disc \Theta_\disc^i, \nabla_\disc \Phi_\disc) \nonumber \\
	&+\mathcal{B}(\hd \Theta_\disc^i, \nabla_\disc P_\disc \Psi, \nabla_\disc \Phi_\disc-\mathcal{B}(\hd \Theta_\disc^i,\nabla_\disc \Theta_\disc^i,\nabla_\disc \Phi_\disc). \label{contraction}
	\end{align}
	The discrete inf-sup condition of $\mathfrak{A}_{\disc,\Psi}(\cdot,\cdot)$, \eqref{contraction}, the continuity of $\mathcal{B}(\cdot,\cdot,\cdot)$, and \eqref{def.CD} leads to the existence of $\Phi_\disc\in \bfXd$ such that $\norm{\Phi_\disc}{\disc}=1$ and
	\begin{align}
	\hat{\beta}\norm{\mu( \Theta_\disc^1)-\mu(\Theta_\disc^2)}{\disc}&\le 	\mathfrak{A}_{\disc,\Psi}(\mu(\Theta_\disc^1)-\mu(\Theta_\disc^2),\Phi_\disc)\nonumber\\
	&=\mathcal{B}(\hd \Theta_\disc^2-\hd \Theta_\disc^1, \nabla_\disc \Theta_\disc^1-\nabla_\disc P_\disc \Psi, \nabla_\disc \Phi_\disc)\nonumber\\
	&\quad+\mathcal{B}( \hd \Theta_\disc^2-\hd P_\disc \Psi,\nabla_\disc \Theta_\disc^2-\nabla_\disc \Theta_\disc^1,\nabla_\disc \Phi_\disc)\nonumber\\
	&\lesssim C_\disc^2\norm{ \Theta_\disc^2-\Theta_\disc^1}{\disc}\Big(\norm{ \Theta_\disc^1-P_\disc \Psi}{\disc}+\norm{ \Theta_\disc^2- P_\disc \Psi}{\disc}\Big).\nonumber
	\end{align}
	Since $\Theta_\disc^1, \Theta_\disc^2 \in B_R(P_\disc \Psi)$ and $C_\disc \le \Gamma$, for a choice of $R$ as in the proof of Theorem \ref{thm.map}, we obtain
	$$\norm{\hd\mu( \Theta_\disc^1)-\hd\mu(\Theta_\disc^2)}{}\le C R\norm{\hd \Theta_\disc^1-\hd \Theta_\disc^2}{},$$
	where $C$ depends only on $\Gamma$ and $\hat{\beta}$. This completes the proof.
\end{proof}
\subsection{Convergence of Newton's Method}\label{sec.Newton}
This section establishes the quadratic convergence of Newton’s method, following an approach similar to that in \cite{ng2}, within the HDM framework.
\medskip

The discrete solution is computed using Newton's method. The iterates $\Psi_\disc^j$ are defined by
\begin{align}\label{eqn.Newton}
\cA(\hessian_\disc \Psi_\disc^j,\hessian_\disc \Phi_\disc)+&\cB(\hessian_\disc \Psi_\disc^{j-1},\nabla_\disc \Psi_\disc^j,\nabla_\disc \Phi_\disc)+\cB(\hessian_\disc \Psi_\disc^{j},\nabla_\disc \Psi_\disc^{j-1},\nabla_\disc \Phi_\disc) \nonumber\\
&=\cL(E_\disc \Phi_\disc)+\cB(\hessian_\disc \Psi_\disc^{j-1},\nabla_\disc \Psi_\disc^{j-1},\nabla_\disc \Phi_\disc) \fl \Phi_\disc \in \bfXd.
\end{align}

\begin{theorem}[Convergence of Newtons method]
Let $\Omega$ be a convex domain and let $\Psi \in \bH^3(\Omega)\cap \bfX$ be a regular solution to \eqref{abstract_weak}. Let $\Psi_\disc$ solves \eqref{abstract_HS}. Then there exists a constant $\eta$, independent of $\rho$, such that for any initial guess $\Psi_\disc^0$ satisfying $\|\Psi_\disc - \Psi_\disc^0\|_\disc \le \eta$, it holds that $\|\Psi_\disc-\Psi_\disc^j\|_\disc \le \eta$ for all $j=0,1,2,\ldots$. Moreover, the Newton iterates defined in \eqref{eqn.Newton} are well-defined and converge quadratically to $\Psi_\disc$.
\end{theorem}
\begin{proof}
Arguments analogous to Lemma~\ref{lemma.discrete.pert} establish a positive constant $\epsilon$ (sufficiently small) independent of $\rho$ such that for each $Z_\disc \in \bfXd$ with $\|Z_\disc -P_\disc \Psi \|_\disc \le \epsilon$, the bilinear form
\begin{equation}\label{eqn.bilinearinfsup}
\cA(\hessian_\disc \cdot, \hessian_\disc \cdot)+\cB(\hessian_\disc Z_\disc, \nabla_\disc \cdot, \nabla_\disc \cdot)+\cB(\hessian_\disc \cdot, \nabla_\disc Z_\disc, \nabla_\disc \cdot)
\end{equation}
satisfies inf-sup condition on $\bfXd \times \bfXd$ with inf-sup constant $\hat{\beta}/4$. 
\smallskip

Theorem~\ref{thm.err} shows $\|P_\disc \Psi-\Psi_\disc\|_\disc \le C\rho$ for  sufficiently small $\rho$. Thus, a choice of $\rho$ small enough ensures $\|P_\disc \Psi-\Psi_\disc\|_\disc \le \epsilon/2$.  Define
\[\eta=\min\{\epsilon/2,\hat{\beta}/(8\|\cB\|C_\disc^2)\}.\]
If the initial guess $\Psi_\disc^0$ satisfies $\|\Psi_\disc-\Psi_\disc^0\|_\disc \le \eta$, then
\[\|P_\disc \Psi-\Psi_\disc^0\|_\disc \le \|P_\disc \Psi-\Psi_\disc \|_\disc +\|\Psi_\disc-\Psi_\disc^0\|_\disc \le \epsilon.\]
The result is established by induction. Assume $\|\Psi_\disc-\Psi_\disc^{j-1}\|_\disc \le \eta$ and $\|P_\disc \Psi-\Psi_\disc^{j-1}\|_\disc \le \epsilon$ holds for some $j \in \mathbb{N}$. Then, the substitution $Z_\disc := \Psi_\disc^{j-1}$ into \eqref{eqn.bilinearinfsup} yields a discrete inf-sup condition. This guarantees a unique solution $\Psi_\disc^j$ at step $j$ of the Newton scheme. The discrete inf-sup condition \eqref{eqn.bilinearinfsup} implies the existence of $\Phi_\disc \in \bfXd$ such that $\|\Phi_\disc\|_\disc=1$ and 
\begin{align*}
\frac{\hat{\beta}}{4}\|\Psi_\disc-\Psi_\disc^j\|_\disc&\le \cA(\hessian_\disc (\Psi_\disc-\Psi_\disc^j), \hessian_\disc \Phi_\disc)+\cB(\hessian_\disc \Psi_\disc^{j-1}, \nabla_\disc (\Psi_\disc-\Psi_\disc^j), \nabla_\disc \Phi_\disc)\\
&\qquad +\cB(\hessian_\disc (\Psi_\disc-\Psi_\disc^j), \nabla_\disc \Psi_\disc^{j-1}, \nabla_\disc \Phi_\disc).
\end{align*}
This, \eqref{eqn.Newton}, \eqref{abstract_weak}, the boundedness of $\cB(\cdot,\cdot,\cdot)$, and \eqref{def.CD} reveal
\begin{align*}
\frac{\hat{\beta}}{4}\|\Psi_\disc-\Psi_\disc^j\|_\disc&\le \cA(\hessian_\disc \Psi_\disc, \hessian_\disc \Phi_\disc)+\cB(\hessian_\disc \Psi_\disc^{j-1}, \nabla_\disc \Psi_\disc, \nabla_\disc \Phi_\disc)\\
&\quad+\cB(\hessian_\disc \Psi_\disc, \nabla_\disc \Psi_\disc^{j-1}, \nabla_\disc \Phi_\disc) -\cB(\hessian_\disc \Psi_\disc^{j-1}, \nabla_\disc \Psi_\disc^{j-1}, \nabla_\disc \Phi_\disc)
\\
&\quad-\cL(E_\disc \Phi_\disc)\\
&=-\cB(\hessian_\disc \Psi_\disc, \nabla_\disc \Psi_\disc, \nabla_\disc \Phi_\disc)+\cB(\hessian_\disc \Psi_\disc^{j-1}, \nabla_\disc \Psi_\disc, \nabla_\disc \Phi_\disc)\\
&\qquad +\cB(\hessian_\disc \Psi_\disc, \nabla_\disc \Psi_\disc^{j-1}, \nabla_\disc \Phi_\disc)-\cB(\hessian_\disc \Psi_\disc^{j-1}, \nabla_\disc \Psi_\disc^{j-1}, \nabla_\disc \Phi_\disc)\\
&=\cB(\hessian_\disc (\Psi_\disc-\Psi_\disc^{j-1}), \nabla_\disc (\Psi_\disc^{j-1}-\Psi_\disc), \nabla_\disc \Phi_\disc)\le \|\cB\|C_\disc^2\|\Psi_\disc-\Psi_\disc^{j-1}\|_\disc^2.
\end{align*}
This results in
\[\|\Psi_\disc-\Psi_\disc^j\|_\disc\le (4\|\cB\|C_\disc^2/\hat{\beta})\|\|\Psi_\disc-\Psi_\disc^{j-1}\|_\disc^2\]
and establishes the quadratic convergence of the Newton method to $\Psi_\disc$. 

\smallskip

 Since $\|\Psi_\disc-\Psi_\disc^{j-1}\|_\disc \le \eta \le \hat{\beta}/(8\|\cB\|C_\disc^2)$, it follows that
\[\|\Psi_\disc-\Psi_\disc^j\|_\disc\le \frac{1}{2}\|\Psi_\disc-\Psi_\disc^{j-1}\|_\disc <\eta\]
and this completes the induction.
\end{proof}

\begin{remark}
The numerical experiments presented in Section \ref{sec:numericalresults} specifically consider data in $\mathbf{L}^2(\Omega)$, where the right-hand side of \eqref{abstract_weak} is substituted with $\mathcal{L}(\Pi_\disc \Phi_\disc)$, see Remark~\ref{rem.L2}. While convergence for data in $\mathbf{H}^{-1}(\Omega)$ is theoretically provable, the numerical construction of the corresponding companion operator is less established and is an ongoing research.
\end{remark}
\section{Examples and Applications}\label{sec:HDMegandapplications}
This section covers HDM examples and the application of the abstract model to the 2D incompressible stream function vorticity formulation of Navier--Stokes (NS) equation and von K\'{a}rm\'{a}n (vK) equations. The corresponding convergence results, specifying the rates of convergence for these applications, are stated at the end of the section.

\subsection{Examples of HDM}\label{sec:HDMeg}

This section deals with examples of schemes that fit into the HDM framework. Let us begin with mesh notation \cite{gdm}.

\smallskip

For any $\cell\in\mesh$, the center of mass of $K$ is defined by $\overline{\x}_\cell$, $|\cell|>0$ denote the measure of $\cell$,  and $n_K$ be the outer unit normal to $K$. Let $\edges$ be the set of all edges of the mesh and the measure of $\edge\in\edges$ be denoted by $|\edge|$. Let the set of vertices in $\mesh$ be denoted by $\vertices$. Let $\mathcal V_{\rm int}$ (resp. $\mathcal V_{\rm ext}$) denote the set of internal vertices of $\mesh$ (resp. vertices on $\partial \O$). The meshes are assumed to be regular \cite{ciarlet1978finite} in the classical sense that the ratio of the diameter and the radius of the largest ball centered at $\overline{\x}_\cell$ and included in $\cell$ is uniformly bounded with a bound independent of the mesh-size $h$.

Let  $\ell\ge 0$ be an integer and $\cell \in \mesh$. Let the space of polynomials of degree at most $\ell$ in $\cell$ be denoted by $\mathbb{P}_\ell(\cell)$ and let $\mathbb{P}_\ell(\mesh)$ be the broken polynomial space.

\subsubsection{\bf Conforming FEM}\label{sec:conf}
For conforming FEM, the finite dimensional space $V_h$ is a subspace of $H^2_0(\O)$. A HD is defined by $$X_{\disc,0}=:V_h, \mbox{  and for }v_\disc\in X_{\disc,0},\,\Pi_\disc v_\disc=v_\disc,\,\nabla_\disc v_\disc=\nabla v_\disc, \,\hd v_\disc=\hessian v_\disc.$$ 
Classical $C^1$ elements that are used for the approximation the solution of fourth order elliptic problems are the Argyris and Bogner--Fox--Schmit finite elements, see \cite{ciarlet1978finite} for more details.
\begin{lemma}\cite{HDM_nonlinear}\label{th.conformingFEM}
	Let $\disc$ be a HD for the conforming FEM. Then the following hold:
	\begin{itemize}
		\item[(i)] (Coercivity) $C_\disc {\lesssim 1}$,
		\item[(ii)] (Consistency) $ \fl \varphi \in H^{3}(\O)\cap H^2_0(\O)$, $S_{\disc}(\varphi) \lesssim h\norm{\varphi}{3}$,
		\item[(iii)] (Limit-conformity) $\fl \xi \in H^2(\O;\R^{2 \times 2}), \fl \phi \in H^1(\O;\R^2),$  $$W_{\disc}(\xi)=0, \quad \hat{W}_{\disc}(\phi)=0.$$
		\item[(iv)] (Compactness) For a sequence of meshes $(\mesh_{h_m})_{m \in \N}$ with $h_m \rightarrow 0$, denoting the HD constructed on $\mesh_{h_m}$ by $\disc_m$, the sequence $(\disc_m)_{m \in \N}$ is compact.
\item[(v)] (Companion operator)  $\delta (E_\disc)=\omega(E_\disc)=0,$ and $\Gamma(E_\disc)=1.$
	\end{itemize}
\end{lemma}	
\subsubsection{\bf Non-conforming FEMs}\label{sec:ncfem}

Two ncFEMs in dimension $d=2$, namely the Morley FEM and the Adini FEM, fit into the framework of HDM and are discussed below.

\medskip

\textsc{(i) The Morley element \cite{ciarlet1978finite}:} 
Let $
 \llbracket \phi \rrbracket$ be the jump of $\phi$ across the edges. The nonconforming Morley element space associated with the mesh $\mesh$ is defined by
\begin{align*}
V_h & =: \bigg\{\phi \in \Poly_2(\mesh) |\phi \mbox{ is continuous at } \mathcal V_{\rm int}\mbox{ and vanishes at }\mathcal V_{\rm ext},\, \\ &\qquad\fl \edge \in \edgesint, \, \int_{\edge}^{}\bigg\llbracket\frac{\partial \phi}{\partial n}\bigg\rrbracket ds=0;\,\fl \edge \in \edgesext, \, \int_{\edge}^{}\frac{\partial \phi}{\partial n}ds=0\bigg\}.
\end{align*}
On each triangle, the local degrees of freedom are the values of the function at each vertex and the values of the normal derivatives at the midpoints of edges. 
\smallskip

\textbf{HD for the Morley triangle \cite{DS_HDM}: }Each $v_\disc\in X_{\disc,0}$  is a vector of degrees of freedom at the vertices of the mesh (with zero values at boundary vertices) and at the midpoint of the edges opposite to these vertices (with zero values at the midpoint of the boundary edges). The function $\Pi_\disc v_\disc$ is such that $(\Pi_\disc v_\disc)_{|K}\in \mathbb{P}_2(\cell)$ and $\Pi_\disc v_\disc$ (resp. its normal derivatives) takes the values at the vertices (resp.\ at the edge midpoints) dictated by $v_\disc$, $\nabla_\disc v_\disc= \nabla_{\mesh}(\Pi_\disc v_\disc)$ is the broken gradient and $\hd v_\disc= \hessian_{\mesh}(\Pi_\disc v_\disc)$ is the broken Hessian.
\medskip

\textsc{(ii) The Adini element \cite{ciarlet1978finite}:}
The Adini finite element space is the subspace of $H^1_0(\O) \cap C^0(\overline{\O})$ defined by
\begin{align*}
V_h  =: \{&v_h \in L^2(\O);\, v_h\rvert_\cell \in \mathbb{P}_\cell \,\forall \,\cell \in \mesh, v_h \mbox{ and }\nabla v_h \mbox{ are continuous }\\ &\mbox{ at the vertices $\vertices$, and vanish at the vertices in }\vertices_{\rm ext}\},
\end{align*}
where $\mathbb{P}_\cell := \Poly_3(K) \oplus \{x_1x^3_2\} \oplus\{x^3_1x_2\}$. The set of degrees of freedom in each cell are the values of function and all first order derivatives at each vertex. 

\textbf{HD for the Adini rectangle \cite{HDM_linear}: }Each $v_\disc\in X_{\disc,0}$ is a vector of three values at each vertex of the mesh (with zero values at boundary vertices), corresponding to function and gradient values, $\Pi_\disc v_\disc$ is the function such that $(\Pi_\disc v_\disc)_{|K}\in \mathbb{P}_\cell$ and its derivatives take the values at the vertices dictated by $v_\disc$, $\nabla_\disc v_\disc=\nabla(\Pi_\disc v_\disc)$ and $\hd v_\disc= \hessian_{\mesh}(\Pi_\disc v_\disc)$.

\begin{lemma}\cite{HDM_nonlinear, Morley_plate, Brenner_ncfem}\label{th.MorleyAdini}
	Let $\disc$ be a HD for the Morley (resp.\ Adini) ncFEM. Then the following hold:
	\begin{itemize}
		\item[(i)] (Coercivity) $C_\disc {\lesssim 1}$,
		\item[(ii)] (Consistency) $ \fl \varphi \in H^{3}(\O)\cap H^2_0(\O)$, $S_{\disc}(\varphi) \lesssim h\norm{\varphi}{3}$,
		\item[(iii)] (Limit-conformity) $\fl\xi \in H^2(\O;\R^{2 \times 2}), \fl \phi \in H^1(\O;\R^2),$ 
		$$
		W_{\disc}(\xi)+\hat{W}_{\disc}(\phi)\lesssim h\big( \norm{\xi}{2} + \norm{\phi}{1}\big),
		$$
		\item[(iv)] (Compactness) For a sequence of meshes $(\mesh_{h_m})_{m \in \N}$ with $h_m \rightarrow 0$, denoting the HD constructed on $\mesh_{h_m}$ by $\disc_m$, the sequence $(\disc_m)_{m \in \N}$ is compact.
\item[(v)] (Companion operator)  $\delta (E_\disc) \lesssim h^2,\, \omega(E_\disc)\lesssim h,$ and $\Gamma(E_\disc)\lesssim1.$
	\end{itemize}
\end{lemma}	

\subsubsection{\bf Method based on GR Operators}\label{sec.grmethod}

This method \cite{BL_FEM} is attractive as the approximation space consists of continuous piecewise linear functions and the discrete Hessian is constructed by using a GR operator.
\smallskip

Let $(V_h,Q_h,I_h,\stab_h)$ be a quadruplet of a finite element space $V_h\subset H^1_0(\O)$, a projector $Q_h:L^2(\O)\to V_h$, an interpolant $I_h:{H^2_0(\O)}\to V_h$ and a function $\stab_h\in L^\infty(\O;\R^d)$ that stabilises the reconstructed Hessian such that \cite{HDM_linear}
\begin{itemize}
	\item[\propty{0}] [\emph{Strucure of $V_h$ and $I_h$}] For all $z \in V_h$, $\norm{\nabla z}{}\lesssim h^{-1}\norm{z}{}$
	and, for $\varphi \in H^2_0(\O)$, 
	$\norm{\nabla I_h \varphi-\nabla \varphi}{} \lesssim h \norm{\varphi}{2}$. 
	\item[\propty{1}] [\emph{Stability of $Q_h$}] For $\phi \in L^2(\O)$,  $\norm{Q_h \phi}{}
	\lesssim  \norm{\phi}{}.$ 
	\item[\propty{2}] [\emph{$Q_h\nabla I_h$ approximates $\nabla$}] 
	There exists $W$ densely embedded in $H^3(\Omega) \cap H^2_0(\Omega)$ such that $\norm{Q_h\nabla I_h \psi-\nabla \psi}{} \lesssim h^2\norm{\psi}{W}$ for all $\psi\in W$.
	\item[\propty{3}] [\emph{$H^1$ approximation property of $Q_h$}] For $w \in {H^2(\O)\cap H^1_0(\O)}$, 
	$\norm{\nabla Q_hw-\nabla w}{} \lesssim h\norm{w}{2}.$
	\item[\propty{4}] [\emph{Asymptotic density of $[(Q_h\nabla-\nabla)(V_h)]^\bot$}] Setting $N_h=[(Q_h\nabla-\nabla)(V_h)]^\bot$, where the orthogonality is with respect to  the $L^2(\O;\R^d)$-inner product, the following approximation property holds:
	\[
	\inf_{\mu_h\in N_h}\norm{\mu_h-\varphi}{}\lesssim h\norm{\varphi}{1}\quad \fl  \varphi \in H^1(\O;\R^d).
	\]
	\item[\propty{5}] [\emph{Stabilisation function}] $1\le |\stab_h| \lesssim 1$ and, for all $\cell \in \mesh$, 
	$$\left[\stab_{h|K}\otimes (Q_h\nabla-\nabla)(V_h(\cell))\right] \perp \nabla V_h(\cell)^d,$$
	where $V_h(\cell)=\{v_{|\cell}\,:\,v \in V_h\,,\; \cell \in \mesh\}$ and the orthogonality is understood in $L^2(K;\R^{d \times d})$ with the inner product
	induced by ``$:$''.
\end{itemize}

\textbf{HD using GR \cite{HDM_linear}: }The HD based on a quadruplet $(V_h,Q_h,I_h,\stab_h)$ satisfying \propty{0}--\propty{5}
	is defined by: $X_{\disc,0}=V_h$ and, for $u_\disc\in X_{\disc,0}$,
	\[
	\Pi_\disc u_\disc=u_\disc\,,\;\nabla_\disc u_\disc= Q_h\nabla u_\disc\;\mbox{ and }\;\hd u_\disc=\nabla (Q_h \nabla u_\disc)+\stab_h\otimes (Q_h\nabla u_\disc-\nabla u_\disc).
	\]

\begin{lemma}\cite{HDM_nonlinear, DS_HDMcontrol}\label{th.GR}
	Let $\disc$ be a HD such that $(V_h,I_h,Q_h,\stab_h)$ satisfy \propty{0}--\propty{5}. Then,
	\begin{itemize}
		\item[(i)]  (Coercivity) $C_\disc \lesssim 1$, 
		\item[(ii)](Consistency) $ \fl \varphi \in W$, $S_{\disc}(\varphi) \lesssim  h\norm{\varphi}{W}$,
		\item[(iii)] (Limit-conformity) $\fl \xi \in H^2(\O;\R^{d \times d})$ and $\fl \phi \in H_{\rm{div}}(\O)$, $$\hat{W}_{\disc}(\phi)=0, \quad W_{\disc}(\xi) \lesssim h \norm{\xi}{2}.$$
		\item[(iv)] (Compactness) If $(\mesh_m)_{m\in\N}$ is a sequence of meshes and $\disc_m$ is a GR HD based on $\mesh_m$ for discrete elements satisfying \propty{0}--\propty{5} uniformly with respect to $m$, then $(\disc_m)_{m\in\N}$ is compact.
\item[(v)] (Companion operator)  $\delta (E_\disc) \lesssim h,\, \omega(E_\disc)= 0,$ and $\Gamma(E_\disc)\lesssim1.$
	\end{itemize}
\end{lemma}

\subsection{Applications}\label{sec:applications}
This section establishes that the abstract formulation \eqref{abstract_weak} is applicable to the stream function-vorticity form of the 2D incompressible Navier–Stokes equations for $k=1$ and von K\'{a}rm\'{a}n equations for $k=2$.

\subsubsection{Navier--Stokes problem \cite{Lions_NS,BRR}:} For $f \in H^{-s}(\O)$ with $s \in \{1,2\}$, let $u$ solve
\begin{subequations}\label{NS_eq}
\begin{align} 
\Delta^2u + \frac{\partial}{\partial x_1}\bigg((-\Delta u)\frac{\partial u}{\partial x_2}\bigg)&- \frac{\partial}{\partial x_2}\bigg((-\Delta u)\frac{\partial u}{\partial x_1}\bigg)=f \mbox{ in } \O\\
&u=\frac{\partial u}{\partial n}=0 \mbox{ on } \partial\O.
\end{align}
\end{subequations}
 The weak formulation to \eqref{NS_eq} seeks $u \in X$ such that
\be\label{NS_weak}
\mathcal{A}(\hessian u,\hessian v)+\mathcal{B}(\hessian u, \nabla u, \nabla v)=\mathcal{L}(v) \fl  v \in X,
\ee
where for all $\xi,$ $\chi \in L^2(\O;\sym2)$ and $\phi,$ $ \theta \in L^2(\O,\R^2),$
$$\mathcal{A}(\xi,\chi)= \int_{\O}^{}\xi: \chi \d \x, \quad \mathcal{B}(\xi, \phi,\theta)=\int_{\O}^{}\mbox{tr}(\xi) \phi\cdot \mbox{rot}_{\pi/2}(\theta)\d\x,\quad\mathcal{L}(v)=\langle f,v\rangle _{-s,s}. $$
Note that $\mbox{tr}(\xi)$ means the trace of the matrix $\xi$ and, for $\theta=(\theta_1,\theta_2)$,
$\mbox{rot}_{\pi/2}(\theta)=\big(-\theta_2, \theta_1\big)^t.$ If $f \in L^2(\Omega)$, then the linear functional $\cL(v)$ is defined by the integral $\displaystyle \int_{\O}fv d\x$.

\medskip

It is straightforward to verify that $\mathcal{A}(\cdot,\cdot)$, $\mathcal{B}(\cdot,\cdot,\cdot)$, and $\mathcal{L}(\cdot)$ satisfy Assumptions $\assum{1}$-$\assum{4}$ for $k=1$. In particular, the continuity of the trilinear form $\mathcal{B}(\cdot,\cdot,\cdot)$ is a direct consequence of the generalized Hölder's inequality with $\|\cB\|=1$.

\subsubsection{The von K\'{a}rm\'{a}n equations \cite{ciarlet_plate}: } Given $f \in H^{-s}(\Omega)$ where $\Omega \subsetneq \R^2$, seek the vertical displacement  $u$ and the  Airy stress function $v$ such that  
\begin{subequations}\label{vk_eq}
	\begin{align}%
	\Delta^2u&=[u,v]+f \mbox{ in }\O, \label{vk_eq1}\\
	\Delta^2v&=-\frac{1}{2}[u,u] \mbox{ in }\O,\label{vk_eq2}
	\end{align}
\end{subequations}
with clamped boundary conditions
\be \label{clamped_bc}
u=\frac{\partial u}{\partial n}=v=\frac{\partial v}{\partial n}=0 \mbox{ on } \partial\O.
\ee
The von K\'{a}rm\'{a}n bracket $[\cdot,\cdot]$ is defined by
$$[\xi,\chi]=\xi_{xx}\chi_{yy}+\xi_{yy}\chi_{xx}-2\xi_{xy}\chi_{xy}=\mbox{cof}(\hessian\xi):\hessian\chi,$$ where $\mbox{cof}(\hessian \xi)$ denotes the co-factor matrix of $\hessian \xi$. Then a weak formulation corresponding to \eqref{vk_eq}-\eqref{clamped_bc} seeks $u, v \in H^2_0(\O)$ such that
\begin{subequations}\label{vk_weak}
	\begin{align}
	a(u,\phi_1)+2b(u,\phi_1,v)&=(f,\phi_1)_{-s,s} \fl \phi_1 \in X, \label{vk_weak1} \\
	2a(v,\phi_2)-2b(u,u,\phi_2)&=0 \fl  \phi_2 \in X, \label{vk_weak2}
	\end{align}
\end{subequations}
where for all $\xi, \chi, \phi \in X,$
\[
\begin{aligned}
&a(\xi, \chi):=\int_\O \hessian \xi : \hessian \chi\d\x,\quad b(\xi,\chi,\phi):=\frac{1}{2}\int_\O\mbox{cof}(\hessian \xi)\nabla\chi \cdot\nabla \phi\d\x=-\frac{1}{2}\int_\O[\xi,\chi]\phi\d\x.
\end{aligned}
\]
Note that $b(\cdot,\cdot,\cdot)$ is derived using the divergence-free rows property \cite{evans_pde} and is symmetric with respect to all variables.  

\medskip

By summing \eqref{vk_weak1} and \eqref{vk_weak2}, we arrive at an equivalent vector-valued formulation as defined in \eqref{abstract_weak} with $k=2$. This formulation seeks a solution $\Psi=(u,v) \in \mathbf{X}$ such that 
\be  \label{vk_weak_vector}
\mathcal{A}(\hessian \Psi,\hessian \Phi)+\mathcal{B}(\hessian \Psi,\nabla \Psi, \nabla \Phi)=\mathcal{L}(\Phi) \fl \Phi \in \bfX,
\ee
where for all $\Phi=(\phi_1,\phi_2)$, $\Lambda=(\lambda_1,\lambda_2),\,\Gamma=(\gamma_1,\gamma_2),\, \Theta=(\theta_1,\theta_2)$ and $\Xi=(\xi_1, \xi_2)$
with $\Lambda, \Gamma \in 
\bsymb{L}^2(\O;\sym2)$ and $\Xi,$ $\Theta \in \bsymb{L}^2(\O,\R^2),$
	\begin{align*}
	\mathcal{A}(\Lambda, \Gamma):&=\int_\O \lambda_1 :\gamma_1\d\x +2\int_\O  \lambda_2 : \gamma_2\d\x, \\
	\mathcal{B}(\Lambda,\Xi,\Theta):&=\int_\O\mbox{cof}(\lambda_1)\theta_1\cdot\xi_2\d\x- \int_\O\mbox{cof}(\lambda_1)\xi_1 \cdot\theta_2\d\x,\mbox{ and} \\
	\mathcal{L}(\Phi):&=\langle f,\phi_1\rangle_{-s,s}.
	\end{align*}
 This example satisfies Assumptions \assum{1}–\assum{4} directly.
 \begin{remark}\label{remark.vk.wf}
The more commonly used equivalent weak formulation of the von K\'{a}rm\'{a}n model \cite{Brezzi,ng1,ng2} \eqref{vk_eq}  seeks $(u,v) \in \bfX$ such that 
	\begin{align*}
	a(u,\phi_1)+2b(u,v,\phi_1)&=\langle f,\phi_1\rangle_{-s,s} \fl \phi_1 \in X,\\
	a(v,\phi_2)-b(u,u,\phi_2)&=0 \fl \phi_2 \in X .
	\end{align*}
An advantage of \eqref{vk_weak} adopted in \cite{HDM_nonlinear} is that it ensures \assum{3} and hence an a priori bound follows easily without any integration by parts, both at the continuous level and the discrete level.
\end{remark}
\begin{remark}[Order of convergence for NS and vK equations]\label{rem.order}
Let $\O$ be a convex domain and let $u\in X\cap H^3(\Omega)$ solves \eqref{NS_weak} and $u_\disc \in X_{\disc,0}$ solves the corresponding \eqref{abstract_HS}. A combination of Theorem~\ref{thm.err} together with Lemmas~\ref{th.conformingFEM},~\ref{th.MorleyAdini} and~\ref{th.GR} shows the convergence of conforming FEMs, Morley and Adini ncFEMs, and GR methods for NS equations. These methods yields $\mathcal{O}(h)$ convergence rate in $L^2$, broken $H^1$, and broken $H^2$ norms. A similar argument applied to the von K\'{a}rm\'{a}n equations \eqref{vk_weak_vector} yields a linear order of convergence for the same classes of methods.
\end{remark}
\section{Numerical results}\label{sec:numericalresults}

This section presents numerical results for the 2D incompressible NS equation in the stream-function-vorticity formulation and the vK equations using the GR method, Morley ncFEM, and Adini ncFEM. The convergence results for the relative errors in $L^2(\Omega)$, broken $W^{1,4}(\Omega)$, and broken $H^2(\Omega)$ norms for the examples considered in this section have been previously documented in \cite[Section 6]{HDM_nonlinear}; consequently, they are omitted here to avoid redundancy. 
To quantify accuracy, define the relative error in broken $H^1(\O)$ norm as follows:
\begin{align*}
&\err_\disc(\nabla u) :=\frac{\norm{\nabla_\disc u_\disc -\nabla u}{}}{\norm{\nabla u}{}}, 
\end{align*}
where $u$ is the continuous solution and $u_\disc$ is the corresponding HS solution. 

\medskip

In the tables, $h$ denotes the mesh size. The model problem is constructed in such a way that the exact solution is known. The discrete problem is solved using Newton's method, implemented following the approach of \cite{ng2}. The initial guess is chosen to be the solution of the biharmonic part of the corresponding problem, as, e.g., in \cite[Section 7]{CCGMNN18}. In all the examples considered in this paper, Newton's method converges within four iterations. In these cases, the error between the final iterate and the previous iterate is observed to be less than $10^{-9}$. The uniform mesh refinement has been performed using red-refinement criteria. 

\subsection{Navier-Stokes equations}
 
 Consider the unit square domain $\Omega=(0,1)^2$ with the exact solution defined as $u=x^2y^2(1-x)^2(1-y)^2$. The load function $f$ is obtained from
$$\Delta^2 u + \frac{\partial}{\partial x}\bigg((-\Delta u)\frac{\partial  u}{\partial y}\bigg)- \frac{\partial}{\partial y}\bigg((-\Delta u)\frac{\partial u}{\partial x}\bigg)=f.$$

\begin{table}[h!!]
\caption{\small{GR, Morley, and Adini methods, NS equations, convergence results}}
{\small{
\begin{center}
\begin{tabular}{||c|c|c||c|c|c||c|c|c||}
\hline
\multicolumn{3}{||c||}{GR NS} 
& \multicolumn{3}{c||}{Morley NS} 
& \multicolumn{3}{c||}{Adini NS} \\
\hline
$h$ & $\err_\disc(\nabla u)$ & Order 
& $h$ & $\err_\disc(\nabla u)$ & Order
& $h$ & $\err_\disc(\nabla u)$ & Order \\
\hline

0.35355 & 0.567673 & -      
& 1.00000 & 3.565604 & -      
& 0.70711  & 0.202375 & - \\

0.17678 & 0.167145 & 1.7640 
& 0.50000 &1.145805 & 1.6378 
& 0.35355 & 0.092087 & 1.1360 \\

0.08839 & 0.049952 & 1.7425 
& 0.25000 & 0.331570 & 1.7890
& 0.17678 & 0.024378 & 1.9174 \\

0.04419 & 0.013806 & 1.8552 
& 0.12500 &  0.092576 & 1.8406 
& 0.08839 & 0.006146 & 1.9878 \\

0.02210 & 0.003646 & 1.9209 
& 0.06250 &0.024007 & 1.9472
& 0.04419 & 0.001540 & 1.9971 \\

0.01105 & 0.000939 & 1.9575 
& 0.03125 &  0.006062 & 1.9855
&  0.02210& 0.000385 & 1.9993 \\

\hline
\end{tabular}
\end{center}
}}
\label{table.ns_combined}
\end{table}
The errors and orders of convergence for the numerical approximation of $u$ in the broken $H^1$ norm for the GR, Morley, and Adini methods are reported in Table \ref{table.ns_combined}. As observed from the table, all three methods exhibit a quadratic rate of convergence in the $H^1$ norm, whereas the theoretical estimates in Theorem~\ref{thm.err} and Remark~\ref{rem.order} predict only a linear rate of convergence.

\subsection{The von K\'{a}rm\'{a}n equations }

Two domains are considered in this section: a square domain, which is convex, and an L-shaped domain, which is non-convex.

\subsubsection{Square Domain} Choose $u=v=x^2y^2(1-x)^2(1-y)^2$ as the exact solutions on $\Omega=(0,1)^2$. The corresponding load functions are determined by substituting these solutions into the governing equations:  $$f=\Delta^2 u-[u,v]\; \mbox{ and }\;g=\Delta^2 v+\frac{1}{2}[u, u].$$

Tables \ref{table.vke_combined_u}-\ref{table.vke_combined_v} present the relative errors and orders of convergence for the variables $u$ and $v$. For the GR method, quadratic convergence is obtained for $v$, whereas a lower order is observed for $u$. Such behavior was also reported for $W^{1,4}(\Omega)$ in \cite{HDM_nonlinear}. Results for Morley and Adini ncFEM are similar to those of the NS equations. Overall, the numerical convergence rates exceed those predicted by the theoretical analysis.

\begin{table}[h!!]
\caption{\small{GR, Morley, and Adini methods, vK equations, convergence results for $u$, square domain}}
{\small{
\begin{center}
\begin{tabular}{||c|c|c||c|c|c||c|c|c||}
\hline
\multicolumn{3}{||c||}{GR vKE} 
& \multicolumn{3}{c||}{Morley vKE} 
& \multicolumn{3}{c||}{Adini vKE} \\
\hline
$h$ & $\err_\disc(\nabla u)$ & Order 
& $h$ & $\err_\disc(\nabla u)$ & Order 
& $h$ & $\err_\disc(\nabla u)$ & Order \\
\hline
0.35355& 0.567835 & -      
& 1.00000 & 3.564432 & -      
& 0.70711 & 0.202649 & - \\

0.17678 & 0.167636 & 1.7601 
& 0.50000 & 1.145871 & 1.6372 
& 0.35355 & 0.092193 & 1.1363 \\

0.08839& 0.050446 & 1.7325 
& 0.25000 & 0.331537 & 1.7892 
& 0.17678 & 0.024409 & 1.9173 \\

0.04419 & 0.014295 & 1.8192 
& 0.12500 & 0.092576 & 1.8405 
& 0.08839  & 0.006154 & 1.9878 \\

0.02210 & 0.004146 & 1.7858 
& 0.06250 & 0.024008 & 1.9471 
& 0.04419 & 0.001542 & 1.9971 \\

0.01105 & 0.001483 & 1.4828 
& 0.03125 & 0.006062 & 1.9855 
&  0.02210& 0.000386& 1.9993  \\
\hline
\end{tabular}
\end{center}
}}
\label{table.vke_combined_u}
\end{table}
\begin{table}[h!!]
\caption{\small{GR, Morley, and Adini methods, vK equations, convergence results for $v$, square domain}}
{\small{
\begin{center}
\begin{tabular}{||c|c|c||c|c|c||c|c|c||}
\hline
\multicolumn{3}{||c||}{GR vKE}
& \multicolumn{3}{c||}{Morley vKE}
& \multicolumn{3}{c||}{Adini vKE} \\
\hline
$h$ & $\err_\disc(\nabla v)$ & Order 
& $h$ & $\err_\disc(\nabla v)$ & Order 
& $h$ & $\err_\disc(\nabla v)$ & Order \\
\hline
0.35355 & 0.567587 & -&1.00000 & 3.566189 & -&0.70711 & 0.202238 & 2.3059 \\
0.17678 & 0.166900 & 1.7659 & 0.50000 & 1.145773 & 1.6381 & 0.35355 & 0.092034 & 1.1358 \\
0.08839 & 0.049707 & 1.7475 & 0.25000 & 0.331586 & 1.7889 & 0.17678 & 0.024363 & 1.9175 \\
0.04419 & 0.013567 & 1.8733 & 0.12500 & 0.092575 & 1.8407 & 0.08839 & 0.006142 & 1.9878 \\
0.02210 & 0.003417 & 1.9894 & 0.06250 & 0.024007 & 1.9472 & 0.04419 & 0.001539 & 1.9971 \\
0.01105&0.000742 & 2.2037& 0.03125 & 0.006062 & 1.9855 &0.02210 & 0.000385 & 1.9993 \\
\hline
\end{tabular}
\end{center}
}}
\label{table.vke_combined_v}
\end{table}
\subsubsection{L-shaped domain}

Consider the L-shaped domain $\Omega=(-1,1)^2 \setminus\big{(}[0,1)\times(-1,0]\big{)}$. Let the exact singular solution \cite{Grisvard92} in polar coordinates be	\begin{align}\label{eqn.uLshaped}
						u=v=(r^2 \cos^2\theta-1)^2 (r^2 \sin^2\theta-1)^2 r^{1+ \gamma}g_{\gamma,\omega}(\theta),
						\end{align}
						where $ \gamma\approx 0.5444837367$ is a non-characteristic 
						root of $\sin^2( \gamma\omega) =  \gamma^2\sin^2(\omega)$, $\displaystyle \omega=\frac{3\pi}{2}$,  and
				\begin{align*}
g_{\gamma,\omega}(\theta)&=\left(\frac{1}{\gamma-1}\sin ((\gamma-1)\omega\right)-\frac{1}{ \gamma+1}\sin(( \gamma+1)\omega))(\cos(( \gamma-1)\theta)-\cos(( \gamma+1)\theta))\\
					&\quad 	-\left(\frac{1}{\gamma-1}\sin(( \gamma-1)\theta \right)-\frac{1}{ \gamma+1}\sin(( \gamma+1)\theta))
						(\cos(( \gamma-1)\omega)-\cos(( \gamma+1)\omega)).
\end{align*}

\begin{table}[h!!]
\caption{\small{Morley and Adini methods, vK equations, convergence results for $u$, L-shaped domain}}
{\small{
\begin{center}
\begin{tabular}{||c|c|c||c|c|c||}
\hline
\multicolumn{3}{||c||}{Morley vKE} 
& \multicolumn{3}{c||}{Adini vKE} \\
\hline
$h$ & $\err_\disc(\nabla u)$ & Order 
& $h$ & $\err_\disc(\nabla u)$ & Order \\
\hline
0.70711 & 1.985957 & -      
& 0.70711 & 0.230581 & - \\

0.35355& 0.623930 & 1.6704 
& 0.35355 & 0.051208 & 2.1708 \\

0.17678 & 0.181811 & 1.7789 
& 0.17678 & 0.012876 & 1.9917 \\

0.08839 & 0.053249 & 1.7716 
& 0.08839 & 0.003761 & 1.7755 \\

0.04419 & 0.017351 & 1.6178 
& 0.04419 & 0.001332 & 1.4978 \\

0.02210& 0.006560 & 1.4033 
& 0.02210 & 0.000558 & 1.2557 \\
\hline
\end{tabular}
\end{center}
}}
\label{table.vke_lshape_combined_u}
\end{table}
\begin{table}[h!!]
\caption{\small{Morley and Adini methods, vK equations, convergence results for $v$, L-shaped domain}}
{\small{
\begin{center}
\begin{tabular}{||c|c|c||c|c|c||}
\hline
\multicolumn{3}{||c||}{Morley vKE}
& \multicolumn{3}{c||}{Adini vKE} \\
\hline
$h$ & $\err_\disc(\nabla v)$ & Order 
& $h$ & $\err_\disc(\nabla v)$ & Order \\
\hline
0.70711 & 1.293881 & -&0.70711 & 0.190722 & - \\
0.35355 & 0.569137 & 1.1849 &0.35355 & 0.044289 & 2.1064 \\
0.17678 & 0.167686 & 1.7630 &0.17678 & 0.011837 & 1.9037 \\
0.08839 & 0.047896 & 1.8078 & 0.08839 & 0.003683 & 1.6842 \\
0.04419 & 0.015209 & 1.6550 &0.04419 & 0.001328 & 1.4718 \\
0.02210&0.005694&1.4175&0.02210 & 0.000541& 1.2945\\
\hline
\end{tabular}
\end{center}
}}
\label{table.vke_lshape_combined_v}
\end{table}

The errors and rates of convergence are reported in Tables~\ref{table.vke_lshape_combined_u}-\ref{table.vke_lshape_combined_v}, respectively. 
This example is particularly interesting, as the solution exhibits reduced regularity due to 
the corner singularity. Since the domain $\Omega$ is nonconvex, only suboptimal orders of 
convergence in the $H^1$ norm are expected, which is clearly reflected in the table.

\begin{remark}[$L^2$, $H^1$, and $H^2$ convergence results]
The rate of convergence in the energy norm for the above examples is linear, as reported in \cite{HDM_nonlinear}, and is consistent with the predictions of Theorem~\ref{thm.err} and Remark~\ref{rem.order} for the square domain. However, numerical results indicate that the convergence rates in the $L^2$ \cite{HDM_nonlinear} and $H^1$ norms are higher than those predicted by the energy estimate. Therefore, deriving improved $L^2$ and $H^1$ error estimates within the HDM framework remains a topic for future work.
\end{remark}

\section{Conclusion}\label{sec:conclusion}
A unified framework using the HDM for fourth-order semilinear elliptic equations with trilinear nonlinearity and general source, which encompasses conforming FEMs, nonconforming FEMs, and GR methods, is considered in this article. Based on four properties of the HDM and a $C^1$-conforming companion operator, the well-posedness of discrete solutions for general data in $H^{-2}(\Omega)$ is demonstrated, and a robust convergence analysis for rough data in $H^{-1}(\Omega)$ is established without extra-regularity assumptions. Moreover, rigorous error estimates in $L^2$, $H^1$, and $H^2$-like norms on convex domains are derived, and the reliability of Newton’s method for the discrete problem is confirmed. The practical utility of this framework is further validated through application to the stream function vorticity formulation of the 2D incompressible Navier–Stokes and von Kármán equations, where numerical results across various methods, such as the GR method and Adini and Morley ncFEMs, consistently confirm the theoretical estimates.


\subsection*{Acknowledgments}

The author gratefully acknowledges Prof. Neela Nataraj and Prof. J\'erome Droniou for their valuable guidance and insightful suggestions in deriving the error estimates.
\bibliographystyle{abbrv}
\bibliography{Fourth_order_elliptic}

\end{document}